\documentclass[11pt,a4paper,reqno]{amsart}
\usepackage[applemac,utf8]{inputenc}
\usepackage[T1]{fontenc}
\usepackage{ulem}
\usepackage{amsmath,esint,color}
\usepackage{amsthm}
\usepackage{amsfonts}
\usepackage{amssymb}
\usepackage{bigints}
\usepackage{fullpage}
\usepackage[colorlinks=true, pdfstartview=FitV, linkcolor=blue, citecolor=blue, urlcolor=blue]{hyperref}

\date{}
\definecolor{sah}{rgb}{0.66,0.33, 0.04}
\definecolor{adel4}{cmyk}{1,0,0,0}
\definecolor{adel3}{rgb}{0.66,0.33, 0.04}
\definecolor{adel1}{cmyk}{0,0.20,1,0}
\definecolor{adel2}{cmyk}{0,0.40,1,0.30}
\definecolor{adel0}{rgb}{0.99,0.60, 0.30}
\definecolor{trut}{rgb}{0.99,0.80, 0.00}
\definecolor{trus}{rgb}{0.00, 0.50, 0.00}
 \definecolor{trust}{rgb}{0.99, 0.99, 0.80}
\definecolor{MaCouleur}{rgb}{0,0.9,0.3}
\newcommand{\Q}{\mathbb{Q}}
  
\newcommand{\NN}{\mathbb{N}}  
\newcommand{\RR}{\mathbb{R}}  
\newcommand{\dx}{\mathit{dx}}
\newcommand{\M}{\mathcal{M}} 
 \newcommand{\dy}{\mathit{dy}} 
\newcommand{\BMO}{\mathit{BMO}}  
   
\newcommand{\BMOF}{\mathit{BMO_F}}  
 \newcommand{\LMOF}{\mathit{LMO_F}}  
\theoremstyle{plain}
\newtheorem{definition}{Definition}
\newtheorem{theorem}{Theorem}
\newtheorem{proposition}{Proposition}
\newtheorem{lemma}{Lemma}
\newtheorem{remark}{Remark}
\newtheorem{coro}{Corollary}

\def\virgp{\raise 2pt\hbox{,}}

\def\Xint#1{\mathchoice
   {\XXint\displaystyle\textstyle{#1}}%
   {\XXint\textstyle\scriptstyle{#1}}%
   {\XXint\scriptstyle\scriptscriptstyle{#1}}%
   {\XXint\scriptscriptstyle\scriptscriptstyle{#1}}%
   \!\int}
\def\XXint#1#2#3{{\setbox0=\hbox{$#1{#2#3}{\int}$}
     \vcenter{\hbox{$#2#3$}}\kern-.5\wd0}}

\def\av_#1{\Xint-_{#1}}

\usepackage[textmath,displaymath,floats,graphics,delayed]{preview}
  \title[]{On the 2D Isentropic Euler System with Unbounded Initial vorticity}
\author[]{ZINEB HASSAINIA}
\address{IRMAR, Universit\'e de Rennes 1 \\ Campus de Beaulieu \\  35 042 Rennes cedex, France}
 \email{zineb.hassainia@univ-rennes1.fr}
\begin{document}
\begin{abstract}
This paper is devoted to the study of the low Mach number limit for the 2D isentropic Euler system associated to ill-prepared initial data  with slow blow up rate on  $\log\varepsilon^{-1}$. 
We  prove in particular  the strong  convergence to the  solution of the incompressible Euler system when  the  vorticity  belongs to some weighted $BMO$ spaces allowing unbounded functions. The proof is based on the extension of the result of \cite{B-K} to a compressible transport model.
\end{abstract}
\maketitle

\begin{quote}
\footnotesize\tableofcontents
\end{quote}
\section{Introduction}
\quad The equations of motion governing a perfect  compressible fluid evolving in the whole space $\RR^2$ are given by  Euler system:

\begin{equation*}
\left\{ \begin{array}{ll}
\rho(\partial_{t}v+v\cdot\nabla v)+\nabla p =0, \, t\geq 0,\, x\in\RR^2&\\
\partial_{t}\rho+\textnormal{div}\, (\rho v) =0, &\\
(v ,\rho)_{| t=0}=(v_{0},\rho_{0}).
\end{array} \right. 
\end{equation*} 
Here, the vector field $v = (v_1,v_2)$ describes  the velocity of the  fluid particles and the scalar functions $p$ and $\rho>0$ stand  for the pressure and the density, respectively. From now onwards, we shall be concerned only with the isentropic case corresponding to the law $$p=\rho^{\gamma},$$ where the parameter $\gamma>1$ is the adiabatic exponent. 

\quad Following the idea of Kawashima, Makino and Ukai \cite{M-U-K}, this system can be symmetrized by using  the sound speed $c$ defined by
$$
c=2\frac{\sqrt{\gamma}}{\gamma-1}\rho^{\frac{\gamma-1}{2}}.
$$
\quad The main scope of this paper is to deal with  the  weakly compressible fluid and particularly we intend  to  get  a lower bound for the lifespan and justify the convergence towards the incompressible system. But before reviewing the state of the art and  giving a precise statement of our main result  we shall briefly describe the way how to get formally  the weakly compressible fluid. In broad terms, the basic idea consists in writing the foregoing system around the equilibrium \mbox{state 
  $(0,c_0)$:}  let $\varepsilon>0$  be a small parameter called the Mach number and set 
$$
v(t,x)=\bar{\gamma}c_0\varepsilon v_\varepsilon(\varepsilon \bar{\gamma}c_0t,x)\quad \textnormal{and}\quad c(t,x)=c_0+\bar{\gamma}c_0\varepsilon c_\varepsilon(\varepsilon \bar{\gamma}c_0t,x)\quad\textnormal{with}\quad \bar{\gamma}=\frac{\gamma-1}{2}.
$$
Then 
the resulting system will be the following
\begin{equation}\label{EC}
\left\{ \begin{array}{ll}
\partial_{t}v_\varepsilon+v_\varepsilon.\nabla v_\varepsilon+\frac{1}{\varepsilon}\nabla c_\varepsilon+\bar{\gamma}c_\varepsilon\nabla c_\varepsilon =0, &\\
\partial_{t}c_\varepsilon+v_\varepsilon\cdot\nabla c_\varepsilon+\frac{1}{\varepsilon}\textnormal{div}\, v_\varepsilon+\bar{\gamma}c_\varepsilon\textnormal{div}\, v_\varepsilon =0, &\\
(v_\varepsilon ,c_\varepsilon)_{| t=0}=(v_{0,\varepsilon},c_{0,\varepsilon}).
\end{array} \right. \tag{E.C}
\end{equation}
As we can easily observe this system contains singular terms in $\varepsilon$  that might affect dramatically the dynamics when the Mach number is  close to zero. For more details 
 about the derivation of the above model we invite the interested reader to consult the papers \cite{H-S, K-M, D-H} and the references therein. \\

\quad From mathematical point of view this system has been intensively investigated  in the few  last decades. One of the basic problems is  the construction of the solutions $(v_\varepsilon,c_\varepsilon)$ in suitable function spaces with a non degenerate time existence and  most importantly  the asymptotic  behavior for small  Mach number. Formally, one expects the velocity $v_\varepsilon$ to converge to $v$ the solution of the incompressible Euler system given by
\begin{equation}
\label{EI}
\left\{ \begin{array}{ll}
\partial_{t}v+v\cdot\nabla v+\nabla p =0, &\\
\textnormal{div}\,  v =0, &\\
v _{| t=0}=v_{0}.
\end{array} \right.\tag{E.I}
\end{equation} 
\quad As a matter of fact, the singular parts are antisymmetric and do not contribute in the energy estimates built over Sobolev spaces $H^s$. Accordingly, a uniform time existence can be shown by using just the theory of hyperbolic systems, see \cite{K-M2}.  However  it is by no means obvious  that the constructed solutions will converge to the expected incompressible Euler solution and the problem  can be highly non trivial when it is coupled with the geometry of the domain, a fact that we ignore here.  In most of the papers dealing with this recurrent subject there are essentially two   kinds of hypothesis on the initial data: the first  one concerns the  well-prepared case where the initial data are assumed  to be slightly compressible meaning  that  $($div$\, v_{0,\varepsilon},\nabla c_{0,\varepsilon})=O(\varepsilon)$ for $\varepsilon$ close to zero. In this context it can be proved that  the time derivative of the solutions is uniformly bounded and therefore the justification of the incompressible limit follows from  Aubin-Lions compactness lemma. For a complete discussion we refer the reader to the papers of Klainerman and Majda \cite{K-M2,K-M}.
The second class of initial data is the ill-prepared case where  the family $(v_{0,\varepsilon}, c_{0,\varepsilon})_{\varepsilon}$ is assumed to be bounded in Sobolev spaces $H^s$ with $s>2$ and the incompressible parts of $(v_{0,\varepsilon})_\varepsilon$ converge strongly to some divergence-free vector field $v_0$ in $L^2$.
In this framework, the main difficulty  that one has to face, as regards the incompressible limit, is the propagation of the time derivative $\partial_t v_\varepsilon$ with the speed $\varepsilon^{-1}$, a phenomenon which does not occur in the case of the well-prepared data. To deal with this trouble,  Ukai used in  \cite{U} the dispersive effects generated by the acoustic waves in order to prove that the compressible part of the velocity and the acoustic term vanish when $\varepsilon$ goes to zero. Similar studies but  in more complex situations and for various models were accomplished later \mbox{in  }different works and for the convenient of the reader  we quote here a short list of \mbox{references
\cite{az,A,Ga, H,H-S, L,L-M,M-S}.}

\quad Regarding the lifespan of these solutions, it is well-known that in contrast to the incompressible case where the classical solutions are global in dimension two, the compressible Euler system (\ref{EC}) may develop singularities in finite time for some smooth initial data. This was  shown in space dimension two by Rammaha \cite{R}, and by Sideris \cite{S} for dimension three. It seems that in dimension two we can generically get  a lower bound for the lifespan $T_\varepsilon$ that goes to infinity for small $\varepsilon$. More precisely,  when the initial data are bounded  in $H^s$ with $s>2$ then  by taking benefit of the vorticity structure coupled with Strichartz estimates we get,
$$
T_\varepsilon  \ge C\log\log \varepsilon^{-1}.
$$
 Besides, we can get  precise information on the lifespans when the initial data enjoy some   specific structures. In fact, Alinhac \cite{A93} showed that in two-dimensional space and for axisymmetric data the lifespan  is equivalent to $\varepsilon^{-1}$. Also, for the three-dimensional system, Sideris \cite{S2} proved the almost global existence of the solution for potential flows. In other words, it was shown that the lifespan $T_\varepsilon$ is bounded below  by $\exp(c/\varepsilon)$. To end this short discussion we mention that  global existence results  were obtained in \cite{G0, S0} for some restrictive initial data.
 
\quad Recently the incompressible limit to \eqref{EC}  for ill-prepared initial data lying to the  critical  Besov space $B^2_{2,1}$  was carried out  in \cite{H-S}. It was also shown that the  strong convergence occurs in  the space of the initial data. The same program was equally  accomplished in dimension  three in \cite{H} for   the axisymmetric initial data. The fact that the regularity is optimal for the incompressible system will contribute with much more technical difficulties and unfortunately  the perturbation theory cannot be easily adapted. In these studies, the  geometry  of the vorticity is of crucial importance. 
 
\quad In  the contributions cited before, the velocity should be in the Lipschitz class uniformly with respect to $\varepsilon$. This constraint was slightly relaxed  in  \cite{D-H} by allowing the initial  data  to be so ill-prepared in order to permit Yudovich solutions for the incompressible system. Recall that these latter solutions are constructed  globally in time for \eqref{EI}  when the initial    vorticity $\omega_0$ belongs \mbox{to  $L^1\cap L^\infty$,} see  \cite{Y}.  In the incompressible framework the vorticity $\omega$, defined for  a vector field $v = (v_1, v_2)$ by $\omega=\partial_1 v_2-\partial_2 v_1,$  is advected by the flow,
\begin{equation}\label{vo}
\partial_t \omega+v\cdot \nabla\omega=0,\quad \Delta v=\nabla^\perp\omega.
\end{equation}

\quad Working in larger spaces than the Yudovich's one for the system \eqref{vo} and peculiarly  with unbounded vorticity, possibly without uniqueness,  is not in general an easy task and often leads  to more technical complications. Nevertheless, in the last decade   slight  progress were done  and we shall here comment only  some of them which fit with the scope of  this paper.  For a complete list of references we invite the reader to check the papers   \cite{D, D-M, Y2}.  One of the basic result in this subject is due to  Vishik in \cite{V} who gave various results when the vorticity belongs to the class $B_\Gamma$: a kind of functional space  characterized by the slow growth of the partial sum built over  the dyadic Fourier blocks.  The results of Vishik which cover global and local existence with or without uniqueness depending on some analytic properties of $\Gamma$ suffer from one inconvenient: the persistence regularity is not proved and an instantaneous  loss of regularity may happen.  

\quad Recently,  Bernicot and Keraani    proved in \cite{B-K} the global existence and uniqueness without any loss of regularity for the  incompressible Euler system when the initial vorticity is taken in a weighted $BMO$ space called  $LBMO$ and denotes the set of  functions with {\it $\log$-bounded mean oscillations}. This space is strictly larger than $L^\infty$ and smaller than the usual $BMO$ space.

\quad  The main task of this paper is  to conduct  the incompressible limit study for \eqref{EC}  when the limiting system \eqref{EI} is posed for initial data lying in the $LBMO$ space. As we shall discuss  later we will be also able   to generalize the result of \cite{B-K} for more general spaces. To give a clear statement we need to introduce the $LBMO$ space and a precise discussion will be found in the next section. First, take $f : \RR^2\to\RR$ be a locally integrable function.
 We say that $f$ belongs to $BMO$ space if
$$
 \Vert f\Vert_{\BMO}\triangleq\sup_{B\, ball}\fint_B \Big\vert f-\fint_B  f \Big\vert.
 $$
Second,   we say that $f$ belongs to the space   $LBMO$  if  
$$
\Vert f \Vert_{\BMOF}\triangleq\Vert f\Vert_{\BMO}+\sup_{2B_{2}\subset B_{1}}\dfrac{\big\vert\fint_{B_{2}}f-\fint_{B_{1}}f\big\vert}{\ln\Big|\dfrac{\ln r_{2}}{\ln r_{1}}\Big|}<+\infty,
$$
 where the supremum is taken over all the  pairs of balls $B_{2}=B(x_{2},r_{2})$ and $ B_{1}=B(x_{1},r_{1})$ in $\RR^2$ with $0<r_1\leq \frac12$.  We have used the notation
$\fint_B f$ to refer to the average $\frac{1}{\vert B\vert}\int_B f(x)\dx$.

Next we shall  state our main result in the special case of $LBMO$ space and whose extension will be given in  \mbox{Theorem \ref{th4}.} 
 \begin{theorem}\label{thi}
Let  $s,\alpha\in]0,1[$ and   $ p\in]1,2[$. Consider a family of initial data $(v_{0,\varepsilon},c_{0,\varepsilon})_{0<\varepsilon<1}$ such that there exists a constant $C>0$  which does not depend on $\varepsilon$ and verifying 
$$
  \Vert (v_{0,\varepsilon},c_{0,\varepsilon})\Vert_{H^{s+2}} \leq C (\log\varepsilon^{-1})^{\alpha},
  $$
 $$
 \Vert \omega_{0,\varepsilon}\Vert_{L^p\cap LBMO}\leq C.
 $$
Then, the system (\ref{EC}) admits a unique solution $(v_\varepsilon,c_\varepsilon)\in C([0,T_\varepsilon[; H^{s+2})$ with the following properties:
\begin{enumerate}
\item
The lifespan $T_\varepsilon$ of the solution satisfies the lower bound:
$$
 T_\varepsilon \geqslant \log \log \log \varepsilon^{-1}\triangleq \tilde{T}_\varepsilon,
$$
and for all $t\leq \tilde{T}_\varepsilon$ we have
\begin{eqnarray} \label{ge}
\Vert \omega_\varepsilon(t)\Vert_{LBMO\cap L^p} & \leq & C_0e^{C_0 t}.
\end{eqnarray}
  Moreover, the compressible and acoustic parts of the solutions converge to zero: 
    $$
  \lim_{\varepsilon\rightarrow 0}\Vert(\textnormal{div}\, v_\varepsilon, \nabla c_\varepsilon)\Vert_{L^1_{\tilde{T}_\varepsilon}L^\infty}=0.
   $$
\item Assume in addition that $\lim_{\varepsilon\rightarrow 0}\Vert\omega_{0,\varepsilon} -\omega_0\Vert_{L^p}=0$, for some vorticity  $\omega_0\in LBMO\cap L^p$. associated to a divergence-free vector field $v_0$. Then the vortices $(\omega_{\varepsilon})_\varepsilon$ converge strongly to the weak solution  $\omega$ of   (\ref{vo}) associated to the initial data $\omega_0$:  for all $t\in \RR_+$ we have
\begin{equation}\label{lv}
\lim_{\varepsilon\rightarrow 0}\Vert\omega_{\varepsilon}(t) -\omega(t)\Vert_{L^q}=0,\quad \forall q\in[p,+\infty[,
\end{equation}
and
\begin{eqnarray} \label{Eul45}
\Vert \omega(t)\Vert_{LBMO\cap L^p} & \leq & C_0e^{C_0t}.
\end{eqnarray}
The constant $C_0$ depends only on the size of the initial data and does not depend on $\varepsilon.$
\end{enumerate}
 \end{theorem}
Before giving a brief account of the proof, we shall summarize some comments in order to clarify some points in the theorem.
\begin{remark}\label{rqq2}
\begin{enumerate}
 \item Theorem $\ref{thi}$  recovers  the result  stated in \cite{B-K} for $LBMO\cap L^p$ space  according to the estimate $\eqref{Eul45}.$
 \item The estimate $ (\ref{lv})$ can be translated to the velocity according to the  Biot-Savart  law $(\ref{b-s})$  as follows
$$
\lim_{\varepsilon\rightarrow 0}\Vert \mathbb{P}v_\varepsilon -v\Vert_{L^\infty_t W^{1,r}\cap L^\infty}= 0\quad \forall r\in\Big[\frac{2p}{2-p},+\infty\Big[.
 $$
where $\mathbb{P}v_\varepsilon=v_\varepsilon-\nabla \Delta^{-1}\textnormal{div}\, v_\varepsilon$ denotes  the Leray's projector over solenoidal vector fields.

\item 
 We can generically construct a family  $(v_{0,\varepsilon})$ satisfying the assumptions of  Theorem $\ref{thi}.$ In fact, let $v_0$ be  a  divergence-free vector field with $\omega_0=\textnormal{curl}\, v_0\in L^p\cap LBMO$. Take two functions   $\chi, \rho  \in C_0^\infty(\RR^2)$  with  
$$\rho\geq 0\quad\hbox{and} \quad \int_{\RR^2}\rho(x) dx=1.
$$
Denote by $(\rho_k)_{k\in\NN^*}$ the usual mollifiers:
$$
\rho_k(x)\triangleq k^2\rho(k x).$$
For $R>0$, we set 
$$
v_{0,\varepsilon}=\rho_k* \Big(\chi\big(\frac{\cdot}{R}\big)v_0\Big).
$$
By the convolution laws we obtain 
\begin{eqnarray*}
\Vert v_{0,\varepsilon}\Vert_{H^{s+2}}&\leq & Ck^{s+2}R\Vert v_0\Vert_{L^\infty}.
\end{eqnarray*}
We choose carefully $k$ and $R$ with slight growth with respect to $\varepsilon$   in order to get
$$
k^{s+2}R\leq C_0 (\ln\varepsilon^{-1})^\alpha.
$$
The uniform boundedness of $(\omega_{0,\varepsilon})$ in the space $LBMO$ is more subtle and  will be the object of  Proposition $\ref{proptop}$ and Proposition $\ref{proposi2}$ in  the next section.
\end{enumerate}
\end{remark}

Let us now outline the basic ideas for the proof of  Theorem \ref{thi}. It is founded on  two main  ingredients: the first one which is the most relevant and has an interest in itself concerns the persistence of the regularity $LBMO$  for a compressible  transport model governing the vorticity, 
$$
\partial_t \omega_{\varepsilon}+ v_{\varepsilon}\cdot \nabla \omega_{\varepsilon}+\omega_{\varepsilon}\textnormal{div}\, v_{\varepsilon}=0.
$$
In the incompressible case  \eqref{vo}, Bernicot and Keraani have shown recently in  \cite{B-K} the following estimate 
\begin{eqnarray}\label{inc76}
\Vert \omega(t)\Vert_{LBMO\cap L^{p}}&\leq & C \Vert \omega_{0}\Vert_{LBMO\cap L^{p}}\Big(1+ \int_0^t\Vert v(\tau)\Vert_{LL}d\tau\Big).
\end{eqnarray}
Where $LL$ refers to the norm associated to  the log-Lipshitz space.

Our goal consists in extending this result to the compressible model cited before. To do so, we shall proceed in the spirit of the work  \cite{B-K} by following the dynamics of the oscillations and especially understand the interaction between them and how the global mass is distributed. 
However,   the lack of the incompressibility of the velocity and the quadratic structure of the nonlinearity $\omega_{\varepsilon}\textnormal{div}\, v_{\varepsilon}$  will bring more technical difficulties that we should  carefully analyze.  Our result whose extension will be given later  in Theorem \ref{th2} reads as follows,\begin{eqnarray*}
\Vert \omega_\varepsilon(t)\Vert_{LBMO\cap L^p}&\leq &  C\Vert \omega_{0,\varepsilon}\Vert_{LBMO\cap L^p}\Big(1+\int_0^t\Vert v_\varepsilon(\tau)\Vert_{LL}d\tau\Big)\\ &\times &\Big(1+\Vert\textnormal{div}\, v_\varepsilon\Vert_{L^1_tC^s}\int_0^t\Vert v_\varepsilon(\tau)\Vert_{LL}d\tau\Big)e^{C\Vert \textnormal{div}\, v_\varepsilon\Vert_{L^1_tL^{\infty}}}.
\end{eqnarray*} 
From this estimate  the result \eqref{inc76} follows easily by taking $\textnormal{div}\, v_\varepsilon=0.$ Now to   prove this result we shall first filtrate the compressible part and then reduce the problem to the establishment of  a  logarithmic estimate for the composition in  the space $LBMO$ but  with a flow which does not  necessarily preserve  the Lebesgue measure. For this part we follow the ideas of \cite{B-K}. Once this logarithmic estimate is proven we should come back to the real solution and thus we are led  to establish some law products invoking  some weighted {\it LMO} spaces  acting as multipliers of $LBMO$ space. 

\quad The second ingredient of the proof of Theorem \ref{thi} is the use of the  Strichartz estimates which are    an  efficient tool to deal with the so ill-prepared initial data.  As it has  already been mentioned, this fact was  used in \cite{D-H}  for  Yudovich solutions and here we follow the same strategy but with slight modifications for the strong convergence. This is done directly by manipulating the vorticity equation. 

\quad
The remainder of this paper is organized as follows. In the next section we recall basic results about  Littlewood-Paley operators, Besov spaces and gather some preliminary estimates.  We shall also  introduce some functional spaces and prove some of their basic properties. In \mbox{Section 3} we shall examine the regularity of the flow map and establish a logarithmic estimate for the compressible transport model. Section 4 is devoted to some classical  energy estimate for the system (\ref{EC}) and  the corresponding  Strichartz estimates.
In the last section, we generalize the result of \mbox{Theorem \ref{thi}} and  give the proofs. We close this paper with an appendix covering the proof of some technical lemmas.

\section{Functional tool box}
In this  section, we shall recall the definition of the frequency localization operators, some of their elementary properties and the Besov spaces. We will also introduce some function spaces and discuss  few basic results that will be used later.\\
First of all, we  fix some notations that will be intensively used in this paper.
\begin{enumerate}
\item[•] In what follows, $C$ stands for some real positive constant which may be different in each occurrence and $C_0$ a constant which depends on the initial data.
\item[•]For any  $X$ and $Y,$ the notation $X\lesssim Y$ means that there exists a positive universal constant $C$  such that $X \leq CY$.
\item[•] For a ball $B$ and $\lambda>0$, $\lambda B$ denotes the ball that is concentric with $B$ and whose radius is $\lambda$ times the radius of $B$.
\item[•] We will denote the mean value of $f$ over the ball $B$ by
$$
\fint_B f \triangleq\frac{1}{\vert B\vert}\int_B f(x)\dx.
$$
\item[•] For $p \in [1,\infty]$, the notation $L_T^p X$ stands for the set of measurable functions $f:[0,T]\to X$ such that $t \longmapsto \Vert f(t)\Vert_{X}$ belongs to $L^p ([0,T] )$.
\end{enumerate}
\subsection{Littelwood-Paley theory} Let us recall briefly the classical dyadic partition of the unity, for a proof see for instance \cite{C} : there exists two positive radial functions $\chi\in\mathcal{D}(\RR^2)$ and $\varphi\in\mathcal{D}(\RR^2\backslash\{0\})$ such that
\begin{equation*}\label{3}
\forall\xi\in\RR^2,\quad \chi(\xi)+\sum_{q\ge0}\varphi(2^{-q}\xi)=1;
\end{equation*}
\begin{equation*}\label{4}
\forall\xi\in\RR^2\backslash\{0\},\quad \sum_{q\in\mathbf{Z}}\varphi(2^{-q}\xi)=1;
\end{equation*}
\begin{equation*}\label{6}
\vert j-q\vert\ge 2\Rightarrow \quad \textnormal{Supp}\ \varphi(2^{-j}\cdot)\cap \textnormal{Supp}\ \varphi(2^{-q}\cdot)=\varnothing;
\end{equation*}
\begin{equation*}\label{7}
q\ge1\Rightarrow\quad \textnormal{Supp}\ \chi\cap\textnormal{Supp}\ \varphi(2^{-q}\cdot)=\varnothing.
\end{equation*}
For every $u \in\mathcal{ S}'(\RR^2)$ one defines the non homogeneous Littlewood-Paley operators by,
$$
\Delta_{-1}v=\mathcal{F}^{-1}\big(\chi \hat{v}\big),\quad \forall{q}\in \NN \quad \Delta_{q} v=\mathcal{F}^{-1}\big(\varphi(2^{-q}\cdot) \hat{v}\big) \quad \textnormal{and}\quad S_{q}v=\sum_{-1\le j\le q-1}\Delta_{j}v.
$$
Symilarly, we define the homogeneous operators by
$$
\forall{q}\in \mathbb{Z}\quad \dot{\Delta}_{q} v=\mathcal{F}^{-1}\big(\varphi(2^{-q}\cdot) \hat{v}\big)\quad\textnormal{and}\quad \dot{S}_{q}v=\sum_{-\infty\le j\le q-1}\dot{\Delta}_{j}v.
$$
We notice that these operators map continuously  $L^p$ to itself uniformly with respect to $q$ and $p$. Furthermore, one can easily check that for every tempered distribution $v$, we have
	$$
	v=\sum_{q\geq -1}\Delta_qv,
	$$
	and for all $v\in\mathcal{ S}'(\RR^2)/\{\mathcal{P}[\RR^2]\}$ 
	$$
	v=\sum_{q\in\mathbb{Z}}\dot{\Delta}_qv,
	$$
	where $\mathcal{P}[\RR^2]$ is the space of polynomials.\\
	The following lemma (referred in what follows as Bernstein inequalities) describes how the derivatives act on spectrally localized functions.
\begin{lemma}\label{br}
There exists a constant $C > 0$ such that for all $q \in N, k \in \NN, 1\leq a\leq b\leq\infty$ and for every tempered
distribution $u$ we have
$$
\sup_{\vert\alpha\vert\leq k}\Vert\partial^\alpha S_q u\Vert_{L^b}\leq C^k2^{q(k+2(\frac{1}{a}-\frac{1}{b}))}\Vert S_qu\Vert_{L^a},
$$
$$
C^{-k}2^{qk}\Vert \dot{\Delta}_qu\Vert_{L^b}\leq\sup_{\vert\alpha\vert= k}\Vert\partial^\alpha \dot{\Delta}_qu\Vert_{L^b}\leq C^k2^{qk}\Vert \dot{\Delta}_qu\Vert_{L^b}.
$$
\end{lemma}
Based on Littlewood-Paley operators, we can define  Besov spaces as follows. Let $(p, r) \in [1, +\infty]^2$ and $s \in \RR$. The non homogeneous Besov space $B_{p,r}^s$ is the set of tempered distributions $v$ such that
$$
\Vert v\Vert_{B_{p, r}^{s}}\triangleq\Big\Vert\big(2^{qs}\Vert\Delta_{q}v\Vert_{L^p}\big)_{q\in\mathbb{Z}}\Big\Vert_{\ell^{r}(\mathbb{Z})}<+\infty.
$$
The homogeneous Besov space $\dot{B} ^s_{p,r}$ is defined as the set of $\mathcal{ S}'(\RR^2)/\{\mathcal{P}[\RR^2]\}$  such that
$$
\Vert v\Vert_{\dot{B}_{p, r}^{s}}\triangleq\Big\Vert\big(2^{qs}\Vert\dot{\Delta}_{q}v\Vert_{L^p}\big)_{q\in\mathbb{Z}}\Big\Vert_{\ell^{r}(\mathbb{Z})}<+\infty.
$$
We point out that, for a strictly positive non integer real number $s$ the Besov space $B_{\infty,\infty}^s$ coincides with the usual   H\"{o}lder space $C^s$. For $s\in]0,1[$, this means that
$$
\Vert v\Vert_{B^s_{\infty,\infty}}\lesssim \Vert v\Vert_{L^\infty}+\sup_{x\neq y}\frac{\vert v(x)-v(y)\vert}{\vert x-y\vert^s}\lesssim\Vert v\Vert_{B^s_{\infty,\infty}}.
$$
 Also we can identify $B_{2,2}^s$ with the Sobolev space $H^s$ for all $s\in\RR$. \\
The following embeddings are an easy consequence of Bernstein inequalities,
$$
B_{p_1, r_1}^{s}\hookrightarrow B_{p_2, r_2}^{s+2(\frac{1}{p_2}-\frac{1}{p_1})} \quad p_1\leq p_2\quad \textnormal{and}\quad r_1\leq r_2.
$$
Next, we recall the  log-Lipschitz space, denoted by $LL$. It is the set of bounded functions $v$  such that
$$
\Vert v \Vert_{LL}\triangleq \underset {0<\vert x- y\vert<1}{\sup}\frac{\vert v(x)-v(y)\vert}{\vert x-y\vert  \log \frac{e}{\vert x-y \vert} }<+\infty.
$$
Note that the space $B^1_{\infty ,\infty}$ is a subspace of $LL$.  More precisely, we have the following inequality, see \cite{B-C-D} for instance. 

\begin{equation}\label{ll1}
\Vert v\Vert_{LL}\lesssim\Vert\nabla v\Vert_{B_{\infty,\infty}^0}.
\end{equation}
If in addition $v$ is  divergence-free and under sufficient  conditions of integrability, the velocity $v$ is determined by the vorticity $\omega\triangleq \textnormal{rot} v$ by means of the Biot-Savart law 
\begin{equation}\label{b-s}
v(x)=\frac{1}{2\pi}\int_{\RR^2}\frac{(x-y)^\perp}{\vert x-y\vert^2}\omega(y)dy.
\end{equation}
The following result is a deep estimate of harmonic analysis and  related to the singular integrals of  Calder\'on-Zygmund type,
\begin{equation}\label{c-z}
\Vert \nabla v\Vert_{L^p}\leq C\dfrac{p^2}{p-1}\Vert \omega\Vert_{L^p}.
\end{equation}
where C is a universal constant and $p\in]1,\infty[$.\\
\begin{lemma}\label{22}
Let $v$ be a smooth vector field and $\omega$ its vorticity.  Define the compressible part of $v$ by 
$$
\mathbb{Q}v\triangleq \nabla \Delta^{-1}\textnormal{div}\, v.
$$
Then,
\begin{equation}
\Vert v\Vert_{LL} \lesssim \Vert \mathbb{Q} v\Vert_{L^{\infty}}+\Vert\omega\Vert_{L^p\cap B_{\infty,\infty}^0}+\Vert\textnormal{div}\, v\Vert_{B^0_{\infty,\infty}}.
\end{equation}
\end{lemma}
\begin{proof}
The identity $v=\mathbb{P}v+\mathbb{Q}v$ and the  estimate (\ref{ll1})  ensure that
\begin{eqnarray}\label{vll}
\Vert v\Vert_{LL}  &\lesssim & \Vert\nabla\mathbb{P}v\Vert_{B_{\infty,\infty}^0}+\Vert \nabla \mathbb{Q} v\Vert_{B_{\infty,\infty}^0}.\notag \end{eqnarray}
Bernstein inequality, the continuity of $\dot{\Delta}_q\mathbb{P}: L^p \longrightarrow L^p,\,\forall p \in [1,\infty]$ uniformly in $q$ and the classical fact  $\Vert\dot{\Delta}_qv\Vert_{L^\infty}\sim 2^{-q}\Vert\dot{\Delta}_q\omega\Vert_{L^\infty}$ give
\begin{eqnarray}
 \Vert\nabla\mathbb{P}v\Vert_{B_{\infty,\infty}^0} &\leq &\Vert \Delta_{-1}\nabla\mathbb{P} v\Vert_{L^{\infty}}+\sup_{q\in\NN}\Vert\dot{\Delta}_{q}\nabla\mathbb{P} v\Vert_{L^{\infty}}\notag\\ &\lesssim &\Vert \Delta_{-1}\nabla\mathbb{P} v\Vert_{L^{p}}+\sup_{q\in\NN}2^q\Vert\dot{\Delta}_{q}\mathbb{P} v\Vert_{L^{\infty}}\notag\\ &\lesssim & \Vert \nabla\mathbb{P} v\Vert_{L^{p}}+\sup_{q\in\NN}\Vert\dot{\Delta}_{q}\omega\Vert_{L^{\infty}}.\notag
\end{eqnarray}
Since the incompressible part $\mathbb{P}v$ has the same vorticity $\omega$ of  the total velocity:
$$ 
\textnormal{curl }\mathbb{P}v=\omega,
$$  
then by the inequality \eqref{c-z} we get
\begin{equation}
 \Vert\nabla\mathbb{P}v\Vert_{B_{\infty,\infty}^0}\lesssim  \Vert \omega\Vert_{L^{p}}+\Vert\omega\Vert_{B_{\infty,\infty}^0}.\notag
 \end{equation}
 On other hand,
\begin{eqnarray*}
\Vert\nabla\mathbb{Q} v\Vert_{B_{\infty,\infty}^0}&\leq &\Vert \Delta_{-1}\nabla\mathbb{Q} v\Vert_{L^{\infty}}+\sup_{q\in \NN}\Vert\dot{\Delta}_{q}\nabla^2\Delta^{-1}\textnormal{div}\, v\Vert_{L^{\infty}}\\&\leq &\Vert \mathbb{Q} v\Vert_{L^{\infty}}+\Vert\textnormal{div}\, v\Vert_{B^0_{\infty,\infty}}.
\end{eqnarray*}
This concludes the proof of the lemma.
\end{proof}
The following result generalizes the classical Gronwall inequality,  it will be very useful in the proof of the lifespan part of Theorem \ref{th1}. For its proof see Lemma 5.2.1 in \cite{C} .
\begin{lemma}\label{osgood}[Osgood Lemma ]
Let $a, C > 0$, $\gamma : [t_0,T]\rightarrow \RR_+$ be a locally integrable function and  $\mu:[a,+\infty[\longrightarrow \RR_+$ be  a continuous non-decreasing function.  Let  $\rho : [t_0,T]\rightarrow [a,+\infty[$ be a measurable, positive function satisfying
$$
\rho(t)\leq C+\int_{t_0}^t\gamma(s)\mu(\rho(s))ds.
$$
Set $\displaystyle{\mathcal{M}(y)=\int_{a}^{y} \frac{dx}{\mu(x)}}$ and assume that $\displaystyle{\lim_{y\to+\infty}\mathcal{M}(y)=+\infty}$. Then
$$
\forall t\in[t_0,T], \quad \rho(t)\leq \M^{-1}\Big(\M(C)+\int_{t_0}^t\gamma(s)ds\Big).
$$
\end{lemma}

Next, we will introduce a new space which play a crucial role in the study of Euler equations as we will see later.

\subsection{The $\BMOF$ space}The main goal of this section is to introduce the weighted BMO spaces denoted by $BMO_F$ where the function $F$ measures the rate  between  two oscillations. Thereafter, we shall focus on the analysis of some of their useful topological properties.  

\quad To start with, let us recall the classical $\BMO$ space (bounded mean oscillation), which is nothing but   the set of locally integrable function $f:\RR^2\to\RR$ such that
$$
\Vert f\Vert_{\BMO}\triangleq\sup_{B}\fint_{B}\Big\vert f-\fint_{B}f\Big\vert<+\infty,
$$
where $B$ runs over all the balls in $\RR^2$.

It is well-known that the quotient space $BMO$ space by the constants  is a Banach space. Moreover, the $BMO$ space enjoys with the following classical properties:
\begin{enumerate}
\item  BMO  is imbricated between the space of bounded functions and the $ B_{\infty,\infty}^0$ space, that is, $$L^\infty\hookrightarrow BMO\hookrightarrow B_{\infty,\infty}^0.
$$
\item The unit ball of the BMO space is a  weakly compact set.
\end{enumerate}
Another interesting property of the BMO space concerns the rate between two oscillations which is estimated as follows:
take two balls $B_1=B(x_1,r_1)$ and $B_2=B(x_2,r_2)$ such that $2B_2\subset B_1$, then
\begin{equation}\label{eq1}
\Big\vert\fint_{B_{2}}f-\fint_{B_{1}}f\Big\vert\lesssim \ln(1+\frac{r_1}{r_2})\Vert f\Vert_{\BMO}.
\end{equation}
It is worthy pointing out that the local well-posedness theory for the incompressible Euler equations is still an open problem when the initial vorticity belongs to the BMO space. However, it was proved recently in \cite{B-K} that the global well-posedness  can be achieved for the $LBMO$ space, which is larger than the bounded functions and smaller than the BMO space. To be precise this space is defined by
$$
\|f\|_{LBMO}\triangleq\| f\|_{BMO}+\sup_{2B_{2}\subset B_{1}}\dfrac{\big\vert\fint_{B_{2}}f-\fint_{B_{1}}f\big\vert}{\ln\Big|\dfrac{\ln r_{2}}{\ln r_{1}}\Big|}<+\infty,
$$
where the radius $r_1$ of $B_1$ is smaller than $\frac12$. It seems that we can perform the same result for more general spaces by replacing the "outside" logarithm in the second part of the norm by a general function $F$ which must satisfy some special assumptions listed below. We mention that a similar  extension was done in the paper \cite{B-H}. Before stating the definition of these spaces we need the following notions.
\begin{definition}\label{def1}
 Let $F:[1,+\infty[\longrightarrow[1,+\infty[$ be a nondecreasing function.
 \begin{enumerate}
\item[•] We say that $F$ belongs to the class $\mathcal{F}$ if there exits a constant $C>0$ such that the following properties hold true:

\begin{enumerate}
\item Blow up at infinity: $\lim_{x\rightarrow 0}F(x)=+\infty$.
\item {\it Asymptotic behavior}: For any $\lambda\in[1,+\infty[$ and $ x\in[\lambda,+\infty[$, we have
$$
\int_{x}^{+\infty}e^{-\frac{y}{\lambda}}F(y)\dy\leq C\,\lambda\, e^{-\frac{x}{\lambda}}F(x).
$$
\item Polynomial growth: For all $(x,y)\in ([1,+\infty[)^2$
$$
 F(xy)\leq CF(x)F(y).
$$
\end{enumerate}
\item[•] We say that F belongs to the class $\mathcal{F}'$ if it belongs to $\mathcal{F}$ and satisfies
the Osgood condition:
 $$ \displaystyle\int_{1}^{\infty} \dfrac{\dx}{xF(x)}=+\infty.$$
\end{enumerate}
\end{definition}
Now we shall give some elementary facts listed in the following remark.
\begin{remark}\label{rq1}
\begin{enumerate}
\item Note that the condition $\textnormal{(c)}$ in the previous definition  implies that the function $F$ has at most a polynomial growth. More precisely,  there exists $\alpha>0$ such that
\begin{equation}\label{eq}
F(x)\leq Cx^\alpha \quad \forall x\in[1,+\infty[.  
\end{equation}
\item It turns out from the monotony of $F$ that
$$
F(n)e^{-\frac{n+1}{\lambda}}\leq \int_n^{n+1}F(x)e^{-\frac{x}{\lambda}}dx\leq F(n+1)e^{-\frac{n}{\lambda}}.
$$
Thus, we get the equivalence $$
e^{-\frac{1}{\lambda}}\sum_{n\geq N}e^{-\frac{n}{\lambda}}F(n)\leq \int_N^{+\infty}e^{-\frac{x}{\lambda}}F(x)dx\leq e^{\frac{1}{\lambda}}\sum_{n\geq N}e^{-\frac{n}{\lambda}}F(n).
$$
Consequently,  using the asymptotic behavior of $F$ described by the point $(b)$  and making  a change of variable we obtain for all $a,b\in\RR_+^*$,
\begin{equation}\label{eqis}
\frac{1}{C}\sum_{n> N}e^{-\frac{n}{\lambda}}F\Big(\frac{n+a}{b}\Big)\leq \int_N^{+\infty}e^{-\frac{x}{\lambda}}F(\frac{x+a}{b})dx \leq C\lambda e^{-\frac{N}{\lambda}}F\Big(\frac{N+a}{b}\Big).
\end{equation}
\item  Several examples of functions belonging to the class $\mathcal{F}'$ can be found, for instance, we mention:  $x\longmapsto1+\ln^\alpha(x)$ with $0<\alpha\leq 1$; $x\longmapsto 1+\ln\ln(e+x)\ln(x)$. 
 \item $\mathcal{F}$ is strictly embedded in $\mathcal{F}'$ : For any $\beta>0$, the function $x\longmapsto x^\beta$ belongs to the class $\mathcal{F}\backslash \mathcal{F}'$.
\end{enumerate}
\end{remark}
The  $\BMOF$ space is given by the following definition. 
\begin{definition}
Let $F$ in $\mathcal{F}$, we denote by $\BMO_F$ the space of  the locally integrable functions $f:\RR^2\to\RR$ such that 
$$
\Vert f \Vert_{\BMOF}=\Vert f\Vert_{\BMO}+\sup_{B_{1},B_{2}}\dfrac{\big\vert\fint_{B_{2}}f-\fint_{B_{1}}f\big\vert}{F\Big(\dfrac{1-\ln r_{2}}{1-\ln r_{1}}\Big)}<+\infty,
$$
 where the supremum is taken over all the pairs of balls $B_{2}=B(x_{2},r_{2})$ and $ B_{1}=B(x_{1},r_{1})$ in $\RR^2$ with $0<r_1\leq 1$ and $2B_2\subset B_1$.
\end{definition}
In the next proposition we shall deal with some topological properties for the $BMO_F$ spaces.
\begin{proposition}\label{proptop}
The following properties hold true.
\begin{enumerate}
\item The space $\BMOF$ is complete, included in $BMO$ and  containing $L^\infty(\RR^2)$.
\item If $F, G\in \mathcal{F}$ with $F\lesssim G$ then $BMO_F\hookrightarrow BMO_G.$
\item If \quad$ \ln(1+x)\lesssim F(x), \forall x\geq1$, \, then $L^\infty$ is strictly embedded in $BMO_F.$
\item Let $p\in ]1,\infty],$ then $BMO_F\cap L^p$ is weakly compact.
\item For $g \in L^1$ and $f \in BMO_F$ one has
$$
\Vert g*f\Vert_{BMO_F}\leq \Vert g\Vert_{L^1}\Vert f\Vert_{BMO_F}.
$$
\end{enumerate}
\end{proposition}
\begin{proof}

$\hbox{(i)}$ The two embeddings are straightforward. For the completeness of the space we consider  a Cauchy sequence $(u_n)_n$ in $BMO_F$. Since $BMO_F$ is contained in $BMO$ which is complete, then this sequence converges in $BMO$ and then in $L^1_{loc}$ (see \cite{G2} for instance). Using the definition of the second term of the $BMO_F$ norm and the  convergence in $L^1_{loc}$, we get  the convergence  in $BMO_F$.

$\hbox{(ii)}$ The embedding is obvious from the definition of the second term of the $BMO_F$ norm.

$\hbox{(iii)}$ According to the assertion (ii) we have the embedding $LBMO=BMO_{1+\ln}\hookrightarrow BMO_F$. Now, we conclude by the result of  \cite{B-K} where it is proved that the unbounded function defined by 
\begin{displaymath}
x\longmapsto\left\{ \begin{array}{ll}
\ln(1+\ln\vert x\vert)\quad \textnormal{if}\quad \vert x\vert\leq 1 &\\
0 \quad \textnormal{if}\quad \vert x\vert \geq 1.
\end{array} \right.
\end{displaymath} 
belongs to $BMO_{1+\ln}$.\\

$\hbox{(iv)}$  Let $(w_n)$ be a bounded sequence of $BMO_F\cap L^p$, that is
$$
\sup_{n}\|w_n\|_{BMO_F\cap L^p}= M<\infty.
$$ 
We shall prove that up to an extraction we can find a subsequence denoted also by $(w_n)$ which converges weakly to some $w\in BMO_F\cap L^p$.   The bound of $(w_n)_n$  in $L^p$ implies the existence of a subsequence, denoted also by $(w_n)_n$, and a function $w\in  L^p$  such that $(w_n)$ converges weakly in $L^p$ and consequently for all $B=B(x,r)$ we have
\begin{equation}\label{eqwn}
\lim_{n\to +\infty}\fint_B w_n dx =\fint_B w dx.
\end{equation}
In addition we have
$$
\Vert w\Vert_{L^p}\leq  \lim_{n\to \infty}\inf\Vert w_n\Vert_{L^p}\leq M.
$$
Let $B_1$ and $B_2$ be two balls in $\RR^2$ such that  $2B_2\subset B_1$ and $0<r_1\leq 1$ . As $F$ is larger than $1$ we may write 
 \begin{eqnarray*}
\dfrac{\Big\vert \fint_{B_1} w-\fint_{B_2}w\Big\vert}{F\big(\frac{1-\ln r_2}{1-\ln r_1}\big)} &\leq &  \Big\vert\fint_{B_1} (w-w_n)-\fint_{B_2}(w-w_n)\Big\vert+\dfrac{\Big\vert\fint_{B_1} w_n-\fint_{B_2}w_n\Big\vert}{F\big(\frac{1-\ln r_2}{1-\ln r_1}\big)}\\ & \leq &\Big\vert\fint_{B_1} (w-w_n) \Big\vert + \Big\vert\fint_{B_2}(w-w_n)\Big\vert  + M.
 \end{eqnarray*}
Hence, we get from  (\ref{eqwn}) that
 $$
\dfrac{\Big\vert \fint_{B_1} w-\fint_{B_2}w\Big\vert}{F\big(\frac{1-\ln r_2}{1-\ln r_1}\big)} \leq M.
 $$
 Moreover, from the weak compactness in the BMO space we have
 $$
\Vert w\Vert_{\BMO}\leq  \lim_{n\to \infty}\inf\Vert w_n\Vert_{\BMO}\leq M.
$$

$\hbox{(v)}$ The result follows immediately from the identity
 $$
 x\longmapsto \fint_{B(x,r)}(g*f)=\Big(g* \fint_{B(.,r)}f\Big)(x)\quad \forall r>0.
 $$

\end{proof}
\begin{remark}\label{rq5}
By using the H\"{o}lder inequality we observe that for any ball $B$ of radius $r$ we have
$$
\fint_B \vert f- \fint_B f\vert\leq Cr^{-\frac{2}{p}}\Vert f\Vert_{L^p}.
$$
In this respect we will only need  to deal with balls whose radius is smaller than a universal constant, say $1$.
 \end{remark}
The following proposition gives a rigorous justification for the choice of the initial data given by \mbox{Remark \ref{rqq2}}.
\begin{proposition}\label{proposi2}
Let $1<p<2$, $R>1$  and $\chi$ a be smooth  compactly supported function equal to 1 in the neighborhood of the unit ball. Let  $\omega\in \BMOF \cap L^p$ be the vorticity of a divergence-free vector filed $v$. Then
 $$
\Big\Vert \textnormal{rot}\Big(\chi\big(\frac{\cdot}{R}\big) v\Big)\Big\Vert_{\BMOF\cap L^p}\lesssim \big(\Vert \chi\Vert_{L^\infty}+ \Vert \nabla\chi\Vert_{L^\infty}\big)\Vert \omega\Vert_{BMO_F\cap L^p},
$$
and
$$
\lim_{R\rightarrow +\infty}\Big\Vert \textnormal{rot}\Big(\chi\big(\frac{\cdot}{R}\big)v\Big)-\omega\Big\Vert_{L^p}=0.
$$
\end{proposition}
\begin{proof}
It is obvious that  the term $\textnormal{rot}\big(\chi\big(\frac{\cdot}{R}\big) v\big)$ can be splitted into
\begin{equation}\label{rotxv}
\textnormal{rot}\Big(\chi\big(\frac{x}{R}\big) v\Big)(x)= \chi\big(\frac{x}{R}\big)\omega(x)+\frac{1}{R}\nabla^\perp\chi\big(\frac{x}{R}\big)v(x).
\end{equation}
By H\"{o}lder inequality, it follows that
\begin{eqnarray}\label{rotlp}
\Big\Vert\textnormal{rot}\Big(\chi\big(\frac{\cdot}{R}\big) v\Big)\Big\Vert_{L^p} &\leq & \Vert\chi\Vert_{L^\infty}\Vert\omega\Vert_{L^p}+\frac{1}{R}\big\Vert\nabla\chi\big(\frac{\cdot}{R}\big)\big\Vert_{L^2}\Vert v\Vert_{L^\frac{2p}{2-p}}.
\end{eqnarray}
In view of the Biot-Savart law and  the classical Hardy-Littlewood- Sobolev inequality we have
\begin{equation*}
\Vert v\Vert_{L^\frac{2p}{2-p}}\lesssim \Vert\omega\Vert_{L^p}.
\end{equation*}
Inserting this in (\ref{rotlp}) we conclude that
$$
\Big\Vert\textnormal{rot}\Big(\chi\big(\frac{\cdot}{R}\big) v\Big)\Big\Vert_{L^p} \leq  \big(\Vert\chi\Vert_{L^\infty}+\Vert\nabla\chi\Vert_{L^\infty}\big)\Vert\omega\Vert_{L^p}.
$$
For the $BMO$ part  of the norm we use \eqref{rotxv} combined with the embedding  $L^\infty\hookrightarrow  BMO$, 
\begin{eqnarray}\label{int}
\Big\Vert\textnormal{rot}\Big(\chi\big(\frac{\cdot}{R}\big) v\Big)\Big\Vert_{BMO}&\leq &\big\Vert \chi\big(\frac{\cdot}{R}\big)\omega\big\Vert_{BMO}+\frac{1}{R}\Vert\nabla^\perp\chi\big(\frac{\cdot}{R}\big)\cdot v\Vert_{BMO}\notag\\ &\lesssim & \big\Vert \chi\big(\frac{\cdot}{R}\big)\omega\big\Vert_{BMO}+\frac{1}{R}\Vert\nabla\chi\|_{L^\infty}\|v\|_{L^\infty}.\end{eqnarray}
As $1<p<2$, the Biot-Savart law ensures that
\begin{equation*}
\Vert v\Vert_{L^\infty}\lesssim  \Vert \omega\Vert_{L^p\cap L^{2p}}.
\end{equation*}
Recall the classical result  of interpolation, see \cite{G2} for instance:
\begin{equation}\label{inter}
 \Vert \omega\Vert_{L^q}\lesssim  \Vert \omega\Vert_{L^p}^{\frac{p}{q}} \Vert \omega\Vert_{BMO}^{1-\frac{p}{q}}\quad\forall q\in[p,+\infty[.
\end{equation}
Combining the preceding two  inequalities we get 
\begin{eqnarray}\label{bint}
\Vert v\Vert_{L^\infty}&\lesssim &\Vert \omega\Vert_{L^p}^{\frac{1}{2}} \Vert \omega\Vert_{BMO}^{\frac{1}{2}}+\Vert \omega\Vert_{L^p}\notag\\ & \lesssim &  \Vert \omega\Vert_{BMO\cap L^p}.
\end{eqnarray}
To estimate the first term of  the right-hand side term of inequality \eqref{int}   we write 
\begin{eqnarray*}
 \chi\big(\frac{x}{R}\big)\omega(x)-\fint_B  \chi\big(\frac{y}{R}\big)\omega(y)dy &=&\omega(x)\bigg( \chi\big(\frac{x}{R}\big)-\fint_B  \chi\big(\frac{y}{R}\big)dy\bigg)\\
 &+& \bigg(\fint_B  \chi\big(\frac{y}{R}\big)dy\bigg)\bigg(\omega(x)-\fint_B \omega(y)dy\bigg)\\ &+&\fint_B \omega(y)\bigg(\Big(\fint_B  \chi\big(\frac{z}{R}\big)dz\Big)- \chi\big(\frac{y}{R}\big)\bigg)dy.
\end{eqnarray*}
Hence,
\begin{eqnarray*}
\fint_B\Big\vert  \chi\big(\frac{x}{R}\big)\omega(x)dx-\fint_B \chi\big(\frac{y}{R}\big)\omega(y)dy\Big\vert &\leq & 2\fint_{B}\big\vert\omega(x)\big\vert\Big\vert \chi\big(\frac{x}{R}\big) -\fint_{B}\chi\big(\frac{y}{R}\big)dy\Big\vert dx \\ &+&\Big\vert\fint_B \chi\big(\frac{x}{R}\big)dx\Big\vert \fint_{B}\Big\vert \omega-\fint_{B}\omega\Big\vert \\ &\lesssim  &\fint_{B}\fint_{B}\big\vert \omega(x)\big\vert\Big\vert \chi\big(\frac{x}{R}\big) -\chi\big(\frac{y}{R}\big)\Big\vert dy dx+\Vert \chi\Vert_{L^\infty}\Vert \omega\Vert_{BMO}.
\end{eqnarray*}
By the mean value theorem, we get
\begin{eqnarray*}
\fint_{B}\fint_{B}\big\vert \omega(x)\big\vert\Big\vert \chi\big(\frac{x}{R}\big) -\chi\big(\frac{y}{R}\big)\Big\vert dy dx &\leq & \Vert \nabla\chi\Vert_{L^\infty}\fint_{B}\fint_{B}\big\vert \omega(x)\big\vert\Big\vert \frac{x-y}{R}\Big\vert dy dx \\ &\lesssim & r \Vert \nabla\chi\Vert_{L^\infty} \fint_{B}\big\vert \omega(x)\big\vert dx.
\end{eqnarray*}
Using  H\"{o}lder inequality, the facts $p>1$, $r<1$,  and inequality (\ref{inter})  we find
\begin{eqnarray*}
\fint_{B}\fint_{B}\big\vert \omega(x)\big\vert\Big\vert \chi\big(\frac{x}{R}\big) -\chi\big(\frac{y}{R}\big)\Big\vert dy dx &\lesssim & r^{1-1/p} \Vert \nabla\chi\Vert_{L^\infty}\Vert \omega\Vert_{L^{2p}}  \\ &\lesssim &  \Vert \nabla\chi\Vert_{L^\infty} \Vert \omega\Vert_{BMO\cap L^p}.
\end{eqnarray*}
Concerning the the second term of the norm in $BMO_F$ we start with the identity,
\begin{eqnarray*}
\fint_{B_2} \chi\big(\frac{x}{R}\big) \omega(x)dx-\fint_{B_1} \chi\big(\frac{y}{R}\big) \omega(y)dy &=&\fint_{B_2} \chi\big(\frac{x}{R}\big)\bigg(\omega(x)-\fint_{B_2} \omega(y)dy\bigg)\\ &+&\bigg(\fint_{B_2} \chi\big(\frac{x}{R}\big)dx\bigg) \bigg(\fint_{B_2} \omega(y)dy-\fint_{B_1} \omega(y)dy\bigg)\\ &+&\bigg(\fint_{B_1} \omega(y)dy\bigg)\bigg(\fint_{B_2}\chi\big(\frac{x}{R}\big)dx-\fint_{B_1}\chi\big(\frac{x}{R}\big)dx\bigg)\\ &+&\fint_{B_1} \chi\big(\frac{x}{R}\big)\bigg(\Big(\fint_{B_1}\omega(y)dy\Big)- \omega(x)\bigg)dx\\
&\triangleq&\hbox{I}_1+\hbox{I}_2+\hbox{I}_3+\hbox{I}_4.\notag
\end{eqnarray*}
Using the definition of $BMO_F$ space we get
$$
\frac{|\hbox{I}_1|+|\hbox{I}_2|+|\hbox{I}_4|}{F\Big(\dfrac{1-\ln r_{2}}{1-\ln r_{1}}\Big)}\le \|\chi\|_{L^\infty}\|\omega\|_{BMO_F}.
$$
For the remainder term $\hbox{I}_3$  we proceed as follows: we take $x_0\in B_1\cap B_2$ and we use the main value Theorem with $F\geq1$ 
\begin{eqnarray*}
\frac{|\hbox{I}_3|}{F\Big(\dfrac{1-\ln r_{2}}{1-\ln r_{1}}\Big)}&\lesssim& r_1^{-\frac1p}\|\omega\|_{L^{2p}}\Big(\fint_{B_2}\Big|\chi\big(\frac{x}{R}\big)-\chi(\frac{x_0}{R})\Big|dx+\fint_{B_1}\Big|\chi\big(\frac{x}{R}\big)-\chi(\frac{x_0}{R})\Big|dx\Big)\\&\lesssim& r_1^{-\frac1p}\|\omega\|_{L^{2p}}\|\nabla \chi\|_{L^\infty}\big(r_2/R+r_1/R\big)\\
&\lesssim& \|\omega\|_{BMO\cap L^{p}}\|\nabla \chi\|_{L^\infty}.
\end{eqnarray*}
 We have now to check  the strong convergence in $L^p$ space.  By considering the identity (\ref{rotxv}), H\"{o}lder inequality and the fact that $\chi\equiv 1$ in the neighborhood of the unit ball,
\begin{eqnarray*}
\Big\Vert\textnormal{rot}\Big(\chi\big(\frac{\cdot}{R}\big) v\Big)-\omega\Big\Vert_{L^p} &\leq & \big\Vert\chi\big(\frac{\cdot}{R}\big)-1\big\Vert_{L^\infty}\Vert\omega\Vert_{L^{p}(B^c(0,R))}+\frac{1}{R}\Vert\nabla^\perp\chi\big(\frac{\cdot}{R}\big)\Vert_{L^2}\Vert v\Vert_{L^\frac{2p}{2-p}(B^c(0,R))}\\ &\lesssim &\Vert\omega\Vert_{L^{p}(B^c(0,R))}+\Vert\nabla\chi\Vert_{L^\infty}\Vert v\Vert_{L^{\frac{2p}{2-p}}(B^c(0,R))}.
\end{eqnarray*}
Passing to the limit completes the proof of the desired result. 
\end{proof}
\subsection{$LMO_F$ spaces and law products} Here we endeavor to define a functional space whose elements can be served as  pointwise multipliers for $\BMOF$ space.   In this context, we point out  that $BMO$ is stable by multiplication by $LMO\cap L^\infty$ functions, we refer to  \cite{N-Y} for the proof. Where $LMO$ is a subspace of $BMO$, not comparable to $L^\infty$ and equipped with the semi-norm
 $$
\Vert f\Vert_{LMO}\triangleq\sup_{B}\vert\ln r\vert\fint_{B}\Big\vert f -\fint_{B}f\Big\vert.
$$
In this definition $B$ runs over the balls of radius lesser than $1.$
In order to validate a similar result for the $BMO_F$ space, we have to define the following function space. 
\begin{definition}
Let $F$ be in the class $\mathcal{F}$ and $f\in L^{1}_{loc}(\RR^2,\RR)$, we say that $f$ belongs to the $\LMOF$ space if
$$
\Vert f\Vert_{\LMOF}=\sup_{\underset{r\leq 1}{B}}F(1-\ln r)\fint_{B}\Big\vert f -\fint_{B}f\Big\vert+\sup_{\underset{r_1\leq 1}{2B_{2}\subset B_{1}}}F(1- \ln r_1)\Big\vert\fint_{B_{2}}f-\fint_{B_{1}}f\Big\vert<+\infty, 
$$
\end{definition}
We shall prove the following law product.
\begin{proposition}\label{prod}
Let $g\in\BMOF\cap L^p$, with $1\leq p<\infty$, and $f\in \LMOF\cap L^{\infty}$. Then $fg\in\BMOF \cap L^p$ and
$$
\Vert fg\Vert_{\BMOF}\leq C\Vert f\Vert_{\LMOF\cap L^{\infty}}\Vert g\Vert_{\BMOF \cap L^p},
$$
where $C$ is independent of $f$ and $g$.
\end{proposition}
\begin{proof}
In view Remark \ref{rq5} we will consider throughout the proof $B$ a ball of radius $r<\frac{1}{4}$. We start with writing the following plain identity
\begin{equation}\label{id}
fg-\fint_B fg=f\bigg(g-\fint_B g\bigg)+\bigg(\fint_B g\bigg)\bigg(f-\fint_B f\bigg)+\fint_B \Big\{f\Big(\big(\fint_B g\big)-g\Big)\Big\} ,
\end{equation}
which gives in turn
\begin{eqnarray}\label{bmoest}
\nonumber\fint_B\Big\vert fg-\fint_B fg\Big\vert &\leq & 2\fint_{B}\big\vert f\big\vert\Big\vert g -\fint_{B}g\Big\vert+\Big\vert\fint_B g\Big\vert \fint_{B}\Big\vert f -\fint_{B}f\Big\vert\\ &\lesssim & \Vert f\Vert_{L^{\infty}}\Vert g\Vert_{\BMO}+\Big\vert\fint_B g\Big\vert \fint_{B}\Big\vert f -\fint_{B}f\Big\vert.
\end{eqnarray}
We denote by $\hat{B}$ the ball which is concentric to $B$ and whose radius is equal to $1$.
 According to the definition of the second part of the  $BMO_F$ norm we get
\begin{eqnarray}\label{eqprod}
\Big\vert\fint_B g\Big\vert &\leq & \Big\vert\fint_{B}g -\fint_{\hat{B}}g\Big\vert+\Big\vert\fint_{\hat{ B}}g\Big\vert\notag \\ &\lesssim & F(1-\ln r)\Vert g\Vert_{\BMOF}+\Vert g\Vert_{L^p}.
\end{eqnarray}
It follows that
$$
 \Big\vert\fint_B g\Big\vert \fint_{B}\big\vert f -\fint_{B}f\big\vert \lesssim \Vert g\Vert_{\BMOF\cap L^p}\Vert f\Vert_{\LMOF}.
 $$
Inserting this in \eqref{bmoest}   we find
  $$
 \|f g\|_{BMO}\le C \Vert g\Vert_{\BMOF\cap L^p}\Vert f\Vert_{\LMOF\cap L^\infty}.
 $$
For the second term of the $BMO_F$-norm we will make use of  the following identity
\begin{eqnarray}\label{id2}
\fint_{B_2} fg-\fint_{B_1} fg &=&\fint_{B_2} f\bigg(g-\fint_{B_2} g\bigg)+\bigg(\fint_{B_2} f\bigg) \bigg(\fint_{B_2} g-\fint_{B_1} g\bigg)\\ &+&\bigg(\fint_{B_1} g\bigg)\bigg(\fint_{B_2}f-\fint_{B_1}f\bigg) +\fint_{B_1} f\bigg(\Big(\fint_{B_1}g\Big)- g\bigg).\notag
\end{eqnarray}
 Since $F$ is larger than 1 we may write
$$
\dfrac{\vert\fint_{B_2} fg-\fint_{B_1}fg\vert}{F\Big(\dfrac{1-\ln r_{2}}{1-\ln r_{1}}\Big)} \leq \hbox{ I+II+III+IV},
$$
where
$$
\hbox{I}\triangleq\Big\vert\fint_{B_2}f\big(g-\fint_{B_2}g\big)\Big\vert,
$$
$$
\hbox{II}\triangleq\dfrac{\big\vert\fint_{B_2}f\big\vert\big\vert\fint_{B_2}g-\fint_{B_1}g\big\vert}{F\Big(\dfrac{1-\ln r_{2}}{1-\ln r_{1}}\Big)},
$$
$$
\hbox{III}\triangleq\Big\vert\fint_{B_1}g\Big\vert\Big\vert\fint_{B_2}f-\fint_{B_1}f\Big\vert,
$$
$$
\hbox{IV}\triangleq\Big\vert\fint_{B_1}f(g-\fint_{B_1}g)\Big\vert.
$$

It is  clear that $\hbox{I}$ and $\hbox{IV}$ are bounded by $\Vert f\Vert_{L^{\infty}}\Vert g\Vert_{\BMO}$ and  $\hbox{II}$ is bounded by $\Vert f\Vert_{L^{\infty}}\Vert g\Vert_{\BMOF}$. It remains to estimate $\hbox{III}$. Reproducing the same argument used in \eqref{eqprod} we get
\begin{eqnarray*}
 \hbox{III}&\leq & \Vert g \Vert_{\BMOF\cap L^p}F(1- \ln r_1)\Big\vert\fint_{B_{2}}f-\fint_{B_{1}}f\Big\vert\\ & \leq &  \Vert g \Vert_{\BMOF\cap L^p}\Vert f\Vert_{\LMOF}.
\end{eqnarray*}
This completes the proof of the proposition.
\end{proof}
The next proposition deals with some useful properties of the $\LMOF$ space.
\begin{proposition}\label{alg}
\begin{enumerate}
\item The $\LMOF\cap L^{\infty}$ space  is an algebra. More precisely,
 there exists an absolute constant $C$ such that for any  $f,g\in \LMOF\cap L^{\infty}$ one has
$$
\Vert fg\Vert_{\LMOF\cap L^{\infty}}\leq C\big(\Vert f\Vert_{L^{\infty}}\Vert g\Vert_{\LMOF}+\Vert g\Vert_{L^{\infty}}\Vert f\Vert_{\LMOF}\big).
$$
\item Let $f:\RR\longrightarrow\RR$ be an entire real-function, which vanishes at 0. For any real-valued function $u$ in $LMO_F\cap L^\infty$, the function $f \circ u$ belongs to the same space. Moreover, there exists a positive constant $C$ and an entire real-function $g$ such that we have
$$
\Vert f\circ u\Vert_{\LMOF\cap L^\infty}\leq C\Vert u\Vert_{\LMOF} g\big(\Vert u\Vert_{L^{\infty}}\big).
$$
\item For any $s>0$ we have the embedding $C^s\hookrightarrow \LMOF$,  that is, there exits $C>0$ such that for any $f\in  C^s$,
$$
\Vert f\Vert_{\LMOF}\leq C \Vert f\Vert_{C^s}.
$$
 \end{enumerate}
\end{proposition}
\textbf{Proof}
${\bf{(i)}}$
Making appeal to inequality (\ref{bmoest}) we get
\begin{eqnarray*}
F(1-\ln r)\fint_B\Big\vert fg-\fint_B fg\Big\vert &\leq & 2F(1-\ln r)\fint_{B}\big\vert f\big\vert\Big\vert g -\fint_{B}g\Big\vert+F(1-\ln r)\Big\vert\fint_B g\Big\vert \fint_{B}\Big\vert f -\fint_{B}f\Big\vert\\ &\leq & 2\Vert f\Vert_{L^{\infty}}\Vert g\Vert_{\LMOF}+\Vert g\Vert_{L^{\infty}}\Vert f\Vert_{\LMOF}.
\end{eqnarray*}
Likewise, as $r_2\leq r_1$ and $F$ is a nondecreasing function, we immediately deduce from  identity (\ref{id2}) that
\begin{eqnarray*}
F(1- \ln r_1)\Big\vert\fint_{B_{2}}fg-\fint_{B_{1}}fg\Big\vert &\leq & 
F(1-\ln r_2)\fint_{B_2}\big\vert f\big\vert\Big\vert g -\fint_{B_2}g\Big\vert\\
&+&F(1- \ln r_1)\fint_{B_2}\big\vert f\big\vert \Big\vert \fint_{B_2}g-\fint_{B_1}g\Big\vert\bigg)\\ &+& F(1-\ln r_1)\bigg(\fint_{B_1}\big\vert g\big\vert \Big\vert \fint_{B_2} f -\fint_{B_1}f\Big\vert+\fint_{B_1}\big\vert f\big\vert\Big\vert g -\fint_{B_1}g\Big\vert\bigg)\\ &\leq &
 2\Vert f\Vert_{L^{\infty}}\Vert g\Vert_{\LMOF}+2\Vert g\Vert_{L^{\infty}}\Vert f\Vert_{\LMOF}.
\end{eqnarray*}
 This completes the proof of the assertion (i).\\
${\bf{(ii)}}$ By definition, there exists a sequence $(a_n)_{n}\subset \RR$ such that for all $x\in \RR$,
$$
f(x)=\sum_{n\geq1} a_n x^n
$$
and thus 
$$
f\circ u(x)=\sum_{n=1}^{\infty}a_n u^n(x).
$$
Consequently 
$$
\Vert f\circ u\Vert_{LMO_F\cap L^\infty} \leq \sum_{n=1}^{\infty}\vert a_n\vert\Vert u^n\Vert_{\LMOF\cap L^\infty}.
$$
According to ${\bf{(i)}}$ and using the induction principle, we infer that for all $n\geq2$ one has,
\begin{eqnarray*}
\Vert u^n\Vert_{\LMOF\cap L^\infty}\leq C^{n-1}\Vert u\Vert_{\LMOF}\Vert u\Vert_{L^\infty}^{n-1}.
\end{eqnarray*}
Therefore,
\begin{eqnarray*}
\Vert f\circ u\Vert_{\LMOF\cap L^\infty} &\leq &\Vert u\Vert_{\LMOF}\sum_{n=1}^{\infty}\vert a_n\vert C^{n-1}\Vert u\Vert_{ L^\infty}^{n-1}\\&\leq &\Vert u\Vert_{\LMOF}\sum_{n=0}^{\infty}\vert a_{n+1}\vert C^{n}\Vert u\Vert_{ L^\infty}^n\\ &\triangleq &\Vert u\Vert_{\LMOF} g\big(\Vert u\Vert_{L^\infty}\big).
\end{eqnarray*}
${\bf{(iii)}}$  Using the definition of the H\"{o}lder space, we can write
 \begin{eqnarray*}
F(1-\ln r)\fint_B\big\vert f-\fint_{B} f\big\vert &\leq & F(1-\ln r)\fint_B\fint_B\big\vert f(x)- f(y)\big\vert dx dy\\ &\lesssim & F(1-\ln r)r^s\Vert f\Vert_{C^s}\\
&\lesssim&\Vert f\Vert_{C^s}.
\end{eqnarray*}
The last inequality follows from (\ref{eq}) which implies that
$$
\lim_{r\to 0}F(1-\ln r)r^s= 0. 
$$
The second term of the norm can be handled exactly as the first one. For $x_0\in B_1\cap B_2$ we have
\begin{eqnarray*}
F(1-\ln r_1)\Big\vert\fint_{B_2} f-\fint_{B_1} f\Big\vert &\leq & F(1-\ln r_2)\fint_{B_2}\big\vert f(x)- f(x_0)\big\vert dx \\ &+& F(1-\ln r_1)\fint_{B_1}\big\vert f(x)- f(x_0)\big\vert dx\\ &\leq & \big(F(1-\ln r_2)r_{2}^s+F(1-\ln r_1)r_{1}^s\big)\Vert f\Vert_{C^s},
\end{eqnarray*}
which is bounded for the same reason as before.\\
\section{Compressible transport  model}
We focus in this section on the study of the persistence regularity of the initial data measured in $\BMOF$ space for the following compressible transport model:
\begin{equation}\label{12}
\left\{ \begin{array}{ll}
\partial_t f+ v\cdot \nabla f+f\textnormal{div}\, v=0 ; \\
f_{\vert t=0}=f_0.  \\
\end{array} \right. 
\end{equation}
This  model describes the vorticity dynamics for the system (\ref{EC}) and the advection is governed by a compressible velocity which not necessary in the Lipschitz class.  According to  the inequality \eqref{ll1} the velocity  belongs to the log-Lipschitz class and as it was revealed in the paper \cite{chemin-Bahouri} the solution  in the incompressible case may exhibit a loss of regularity in the classical spaces like  Sobolev and  H\"{o}lder spaces. This possible loss cannot occur in the $LBMO$ space as it was recently proved in  \cite{B-K}. Here we shall extend this latter result to the model \eqref{12} and we will see later an application for the incompressible limit problem.  Before stating our  main result we will  recall the log-Lipschitz norm: \\
$$
\|v\|_{LL}=\sup_{0<|x-y|<1}\frac{|v(x)-v(y)|}{|x-y|\log\frac{e}{|x-y|}}<\infty
$$
 Our result reads as follows.
\begin{theorem}\label{th2}
Let $v$ be a smooth vector field and $f$  be a smooth solution of the system $(\ref{12})$. Then, there exits an absolute constant $C>0$ such that for every $1\leq p\le\infty$,
\begin{equation*}
\Vert f(t)\Vert_{\BMOF\cap L^p}\leq C\Vert f_0\Vert_{\BMOF\cap L^p}e^{C\Vert \textnormal{div}\, v\Vert_{L^1_tL^{\infty}}}F\big(e^{CV(t)}\big)\Big(1+F\big(e^{CV(t)}\big)\Vert\textnormal{div}\, v\Vert_{L^1_t(\LMOF\cap L^\infty)}\Big).
\end{equation*} 
with
$$
V(t)\triangleq \int_0^t\|v(\tau)\|_{LL}d\tau.
$$
\end{theorem}
\begin{remark}
When the velocity is divergence-free the estimate of the preceding theorem becomes
\begin{eqnarray}\label{inc}
\Vert f(t)\Vert_{BMO_{F}\cap L^{p}}\leq C \Vert f_{0}\Vert_{BMO_{F}\cap L^{p}}F\big( e^{CV(t)}\big).
\end{eqnarray}
Therefore, Theorem $\ref{th2}$ recovers  the result stated in \cite{B-K} for $F(x)=1+\ln(x)$
\begin{eqnarray*}
\Vert f(t)\Vert_{LBMO\cap L^{p}}&\leq & C \Vert f_{0}\Vert_{LBMO\cap L^{p}}\Big(1+ \int_0^t\Vert v(\tau)\Vert_{LL}d\tau\Big).
\end{eqnarray*}
As well, we mention that a similar estimate  to $(\ref{inc})$ has been  established  in \cite{B-H} in the incompressible framework  for some $L^\alpha mo_F$ space.
 \end{remark}
To  prove this result,  we shall use  the same approach of \cite{B-H} and \cite{B-K}. However the lack of the incompressibility brings more technical difficulties.
Before going further into the details  we should point out that the solution  of the system (\ref{12}) has, as we will see later, the following structure
$$
f(t,x)=f_0(\psi^{-1}(t,x))\exp\Big(-\int_0^t (\textnormal{div}\, v)(\tau,\psi(\tau,\psi^{-1}(t,x)))d\tau\Big).
$$
Where $\psi$ is the flow associated to the vector field $v$, that is, the solution of the differential equation, 
$$
\partial_t\psi(t,x) = v(t,\psi(t,x)), \quad \psi(0,x)=x.
$$
It turns out that the study of the propagation in the $\BMOF$ space returns to the study of the composition by the flow. Based on that, we shall firstly examine the regularity of the flow map, discuss its left composition with the elements of $\BMOF$  and finally give the proof of Theorem \ref{th2}.
\subsection{The regularity of the flow map} Although the vector field $v$ is not Lipschitz in our context, we still have existence and uniqueness of the flow but a loss of regularity may occur. In fact $\psi$ is not necessary Lipschitz but belongs to the class $C^{s_t}$ with $s_t<1$, as indicated in the lemma below. For more details see for instance \cite{C} and \cite{M}.
 \begin{lemma}\label{p1}
Let $v$ be a smooth vector field on $\RR^2$  and $\psi$ its flow .
Then for all $t \in\RR_+$, we have 
$$
\vert x_1-x_2\vert <e^{-\beta(t)}\Longrightarrow\vert \psi^{\pm 1}(t,x_1)-\psi^{\pm 1}(t,x_2)\vert\leq e\vert x_1-x_2\vert^{\frac{1}{\beta(t)}}.
 $$
 where $\psi^1=\psi$ and $\psi^{-1}$ is the inverse of $\psi.$ The function $\beta(t)$ is given by
 $$\beta(t)=\exp\Big(\int_{0}^{t}\Vert v\Vert_{LL}d\tau\Big).$$
\end{lemma}
As a consequence, we  obtain the  following lemma which is with an extreme importance in the proof of the composition result. Its proof  can be found  in \cite{B-K}.
\begin{lemma}\label{p2}
Under the assumptions of Lemma $\ref{p1}$ and for $r\leq \exp(-\beta(t))$ we have  
$$
4\psi(B(x_0,r))\subset B(\psi(x_0),g_\psi(r)), 
$$
where
\begin{equation}\label{g34}
g_\psi(r)\triangleq4er^{\frac{1}{\beta(t)}} . 
\end{equation}
In particular
\begin{equation}\label{eqg}
\sup\Big\{\frac{1-\ln g_\psi(r)}{1- \ln r},\dfrac{1-\ln r}{1-\ln g_\psi(r)}\Big\}\lesssim 1+\beta(t).
\end{equation}
\end{lemma}
The following inequality will be frequently used in the rest of this paper, see for instance \cite{C}. 
\begin{lemma}\label{lem5}
Let $\psi$ be the flow associated to a smooth vector field $v$. Then for all  $t\in \RR_+$
$$
e^{-\int_0^t\Vert \textnormal{div}\, v(\tau)\Vert_{L^\infty}d\tau}\leq\vert J_{\psi^{\pm 1}_t}(x)\vert\leq e^{\int_0^t\Vert \textnormal{div}\, v(\tau)\Vert_{L^\infty}d\tau} \quad \forall x\in\RR^2.
$$
Where $ J_{\psi_t}(t,x)$ is the Jacobian of $\psi(t,x)$.
\end{lemma}

\subsection{Composition in the $\BMOF$ space }
The problem of the composition in the $BMO$ space can be easily solved when $\psi$ is a bi-Lipschitz map which is unfortunately not necessarily verified in our case. Such difficulty could in general induce a losing regularity but as we will see we can face up this loss by working in a suitable space and replace  $BMO$ space with the   $\BMOF$ spaces.  We will be also led  to  deal with another technical difficulty linked to the fact that $\psi$ is no longer measure-preserving map. Our result is the following,
\begin{theorem}\label{th1}
There exists a positive constant $C$ such that, for any function $f$ taken in $\BMOF\cap L^p$, with $1\leq p\leq\infty$ and for $\psi$ the flow associated to a smooth vector field $v$, we have
$$
\Vert f\circ\psi\Vert_{\BMOF\cap L^p}\leq C\Vert f\Vert_{\BMOF\cap L^p} F\Big(e^{C\Vert v\Vert_{L^1_tLL}}\Big)e^{C \Vert\textnormal{div}\, v\Vert_{L^1_tL^{\infty}}}.
$$
\end{theorem}
\begin{proof}
We know that, by a change of variable, the composition in the $L^p$ space gives
\begin{eqnarray*}
\Vert f\circ \psi\Vert_{L^p}\leq \Vert J_{\psi^{-1}}\Vert_{L^\infty}^{\frac{1}{p}}\Vert f\Vert_{L^p}.
\end{eqnarray*}
The composition in the $\BMOF$ space is more subtle and we shall use the idea of \cite{B-K}. In fact, the proof is divided into two steps: in the first  one we deal with  the $\BMO$ term of the norm and in  the second  we  consider the other term.\\
$\bullet$ \textbf{Step 1: Persistence of the $BMO$ regularity.} We shall start with the persistence of the first part of the norm $BMO_F.$ For this purpose we distinguish  two cases depending whether the radius $r$ is small or not.\\

\underline{\textbf{Case 1}} : $r<(4e)^{-\beta(t)} \min\big\{1,\frac{1}{\Vert J_{\psi}\Vert_{L^{\infty}}}\big\}$. According to the definition \eqref{g34} this condition implies
$$
g_\psi(r)<1.
$$
We denote by $\tilde{B}$ the ball of center $\psi(x_0)$ and radius $g_\psi(r)$. It is easily seen that
\begin{eqnarray*}
\fint_{B}\Big\vert f\circ\psi-\fint_{B}f\circ \psi\Big\vert  &\leq & 2\fint_{B}\Big\vert f\circ \psi -\fint_{\tilde{B}}f\Big\vert. 
\end{eqnarray*}
Then by a change of variable one has
\begin{eqnarray*}
 \fint_{B}\Big\vert f\circ\psi-\fint_{B}f\circ \psi\Big\vert &\leq &\frac{2\Vert J_{\psi^{-1}}\Vert_{L^{\infty}}}{\vert B\vert}\int_{\psi(B)} \Big\vert f -\fint_{\tilde{B}}f\Big\vert \dx.
\end{eqnarray*}
At this stage the strategy  consists in the partition of  the open set $\psi(B)$ into countable balls with variable sizes and to try to measure their interactions with the biggest ball $\tilde{B}$. For this goal we shall use   Whitney covering lemma  \cite{St}  which asserts  in our case  the existence of  a collection of countable open
balls $(O_k)_k$ such that :
\begin{enumerate}
\item[$\bullet$] The collection of double balls is a bounded covering :
$$\psi(B)\subset\underset{k}{\bigcup}2O_k.$$
\item[$\bullet$]The collection is disjoint and for all $k$, 
$$O_k \subset \psi(B).$$
\item[$\bullet$] The Whitney property is verified: the radius $r_k$ of $O_k$ satisfies
$$r_{k}\approx d(O_k,\psi(B)^{c}).$$
\end{enumerate}
So by the first property we may write
\begin{eqnarray*}
\fint_{B}\Big\vert f\circ\psi-\fint_{B}f\circ \psi\Big\vert \dx &\lesssim & \frac{\Vert J_{\psi^{-1}}\Vert_{L^{\infty}}}{\vert B\vert}\sum_{j}\vert O_{j}\vert \fint_{2O_{j}}\Big\vert f-\fint_{\tilde{B}}f\Big\vert \\ &\lesssim &\Vert J_{\psi^{-1}}\Vert_{L^{\infty}} \big( I_{1}+I_{2}\big),
\end{eqnarray*}
 where
$$
\hbox{I}_{1}\triangleq\frac{1}{\vert B\vert}\sum_{j}\vert O_{j}\vert \fint_{2O_{j}}\Big\vert f-\fint_{2O_{j}}f\Big\vert
$$
and
$$
\hbox{I}_{2}\triangleq\frac{1}{\vert B\vert}\sum_{j}\vert O_{j}\vert \Big\vert\fint_{2O_{j}}f-\fint_{\tilde{B}}f\Big\vert. 
$$
Using the fact
\begin{equation}\label{eqn}
\sum_{j}\vert O_{j}\vert \leq  \vert\psi(B)\vert \leq \Vert J_{\psi}\Vert_{L^{\infty}} \vert B\vert,
\end{equation} 
we immediately deduce that
$$
\hbox{I}_{1}\lesssim\Vert J_{\psi}\Vert_{L^{\infty}}\Vert f \Vert_{BMO}.
$$
According to Lemma \ref{p2} we have  $4O_{j}\subset \tilde{B}$. In addition, as $g_\psi(r)<1$ and in view of the definition of the $\BMOF -$norm, we infer that
\begin{eqnarray*}
\hbox{I}_{2} &\lesssim &\frac{1}{\vert B\vert}\sum_{j}\vert O_{j}\vert F\Big(\frac{1-\ln 2r_{j}}{1-\ln g_{\psi}(r)}\Big)\Vert f\Vert_{\BMOF} \\ &\lesssim & \frac{1}{\vert B\vert}\sum_{j}\vert O_{j}\vert F\Big(\frac{1-\ln r_{j}}{1-\ln g_{\psi}(r)}\Big)\Vert f\Vert_{\BMOF}.
\end{eqnarray*}
From \eqref{eqn} we get  $r_{j}\leq \Vert J_{\psi}\Vert_{L^{\infty}}^{1/2}r$ for all $j$. Set
 $$ h(r)\triangleq r\max\big\{1,\Vert J_{\psi}\Vert_{L^{\infty}}\big\},\quad U_k\triangleq \sum_{ e^{-k-1}h(r)<r_j\leq e^{-k}h(r)}\vert O_j\vert,\quad k\in \NN.
$$
Hence as $F$ is non-decreasing we may write
\begin{eqnarray*}
\hbox{I}_{2}&\lesssim &\frac{1}{\vert B\vert}\sum_{k\geq 0}U_{k} F\Big( \frac{k+2-\ln h(r)}{1-\ln g_\psi(r)}\Big)\Vert f\Vert_{\BMOF}\\ &\lesssim &\frac{1}{\vert B\vert}\sum_{k\geq 0}U_{k} F\Big( \frac{k+2-\ln r}{1-\ln g_\psi(r)}\Big)\Vert f\Vert_{\BMOF}.
\end{eqnarray*}
Let $N$ be a fixed positive integer that will be carefully chosen later. We split the preceding sum into two parts
\begin{eqnarray*}
\hbox{I}_2 &\lesssim & \bigg(\frac{1}{\vert B\vert}\sum_{k\leq N}U_k F\Big( \frac{k+2-\ln r}{1-\ln g_\psi(r)}\Big)+ \frac{ 1}{\vert B\vert}\sum_{k> N}U_{k} F\Big( \frac{k+2-\ln r}{1-\ln g_\psi(r)}\Big)\bigg)\Vert f\Vert_{\BMOF}\\ &\triangleq & \big(\hbox{I}_{2,1}+\hbox{I}_{2,2}\big)\Vert f\Vert_{\BMOF}.
\end{eqnarray*}
Since $\sum U_k\leq \Vert J_{\psi}\Vert_{L^{\infty}}\vert B\vert$ and $F$ is non-decreasing then
\begin{equation}\label{eq44}
\hbox{I}_{2,1}\leq \Vert J_{\psi}\Vert_{L^{\infty}} F\Big( \frac{N+2-\ln r}{1-\ln g_\psi(r)}\Big)\Vert f\Vert_{\BMOF}.
\end{equation}
To estimate $\hbox{I}_{2,2}$ we need the following bound of $U_k$ whose proof will  be given in Lemma \ref{ap} of the Appendix.
  $$U_k \lesssim \big(1+\Vert J_{\psi}\Vert_{L^{\infty}}\big)^2 e^{\frac{-k}{\beta(t)}}r^{1+\frac{1}{\beta(t)}}\quad \forall k\geq \beta(t) .$$
Therefore  for $N$ taken larger then $\beta(t)$ we have
\begin{eqnarray*}
\hbox{I}_{2,2}  &\lesssim & \big(1+\Vert J_{\psi}\Vert_{L^{\infty}}\big)^2\sum_{k>N} e^{-\frac{k}{\beta(t)}}r^{\frac{1}{\beta(t)}-1} F\Big(\frac{k+2-\ln r}{1-\ln g_\psi(r)}\Big)\Vert f\Vert_{\BMOF}. 
\end{eqnarray*}
Inequality (\ref{eqis}) from Remark \ref{rq1} gives
\begin{eqnarray}\label{eq5}
\hbox{I}_{2,2}  &\lesssim & \big(1+\Vert J_{\psi}\Vert_{L^{\infty}}\big)^2 e^{-\frac{N}{\beta(t)}}\beta(t)r^{\frac{1}{\beta(t)}-1} F\Big(\frac{N+2-\ln r}{1-\ln g_\psi(r)}\Big)\Vert f\Vert_{\BMOF}. 
\end{eqnarray}
Combining this estimate with (\ref{eq44}) we obtain 
 $$
\hbox{I}_2\lesssim \big(1+\Vert J_{\psi}\Vert_{L^{\infty}}\big)^2\Big(1+ e^{\frac{-N}{\beta(t)}}\beta(t)r^{\frac{1}{\beta(t)}-1}\Big)  F\Big(\frac{N+2-\ln r}{1-\ln g_\psi(r)}\Big)\Vert f\Vert_{\BMOF}.
 $$
Taking $N= \big[\beta(t)(\beta(t)-\ln r)\big]+1$ we get
\begin{eqnarray*}
\hbox{I}_2 &\lesssim &\big(1+\Vert J_{\psi}\Vert_{L^{\infty}}\big)^2\big(1+ e^{-\beta(t)}e^{-\frac{1}{\beta(t)}(1-\ln r)}\beta(t)\big)  F\Big(\frac{\beta(t)(\beta(t)-\ln r)+3-\ln r}{1-\ln g_\psi(r)}\Big)\Vert f\Vert_{\BMOF}\\ &\lesssim &\big(1+\Vert J_{\psi}\Vert_{L^{\infty}}\big)^2  F\Big(\frac{\beta(t)^2+3-(1+\beta(t))\ln r}{1-\ln g_\psi(r)}\Big)\Vert f\Vert_{\BMOF}. 
\end{eqnarray*}
Where we have used in the last inequality  the fact that $\sup_{\beta>1}\beta e^{-\beta}< 1$. Furthermore, from estimate (\ref{eqg}) we have
\begin{eqnarray*}
 F\Big(\frac{\beta(t)^2+2-(1+\beta(t))\ln r}{1-\ln g_\psi(r)}\Big)&\lesssim & F\Big((1+\beta(t))\frac{\beta(t)^2+3-(1+\beta(t))\ln r}{1-\ln r}\Big)\\ &\lesssim & F\big((1+\beta(t))^3\big)\\
 &\lesssim& F(\beta^3(t)).
 \end{eqnarray*}
Hence,
 $$
\hbox{I}_2\lesssim\big(1+\Vert J_{\psi}\Vert_{L^{\infty}}^2\big)F\big(\beta^3(t)\big)\Vert f\Vert_{\BMOF}.
 $$
Putting together the previous estimates gives
\begin{eqnarray*}
 \fint_{B}\Big\vert f\circ\psi -\fint_{B}f\circ \psi\Big\vert &\leq &\|J_{\psi^{-1}}\|_{L^\infty}\big(1+\Vert J_{\psi}\Vert_{L^{\infty}}^2\big)F\big(\beta^3(t)\big)\Vert f\Vert_{\BMOF}. \end{eqnarray*}
According to Lemma \ref{lem5} we find
\begin{equation}\label{conv12}
 \fint_{B}\Big\vert f\circ\psi -\fint_{B}f\circ \psi\Big\vert \le Ce^{C\|\textnormal{div } v\|_{L^1_tL^\infty}}F\big(\beta^3(t)\big)\Vert f\Vert_{\BMOF}.
\end{equation}
  
Let us now move to the second case.\\
 \underline{\textbf{Case 2}} : $1\geqslant r\geq (4e)^{-\beta(t)}\min\big\{1,\frac{1}{\Vert J_{\psi}\Vert_{L^{\infty}}}\big\}$. Under this assumption we can easily check that
 \begin{equation}\label{Eq709}
 \vert\ln r\vert\lesssim \beta(t)+\big\vert \ln \Vert J_{\psi}\Vert_{L^{\infty}}\big\vert .
 \end{equation}
 
By a change of variable  we can write
\begin{eqnarray*}
 \fint_{B}\Big\vert f\circ\psi -\fint_{B}f\circ \psi\Big\vert &\leq & 2\fint_{B}\big\vert f\circ\psi\big\vert \\ &\lesssim & \frac{1}{\vert B\vert}\int_{\psi(B)}\big\vert f(x)\big\vert \vert J_{\psi^{-1}} (x)\vert dx \\ &\lesssim &\frac {1}{\vert B\vert}\sum_{j}\vert O_j\vert \Vert J_{\psi^{-1}}\Vert_{L^{\infty}} \fint_{2O_j}\big\vert f\big\vert \\ &\lesssim & \Vert J_{\psi^{-1}}\Vert_{L^{\infty}}\bigg(\frac{1}{\vert B\vert}\sum_{j}\vert O_{j}\vert\fint_{2O_{j}}\Big\vert f-\fint_{2O_{j}}f\Big\vert +\frac{1}{\vert B\vert}\sum_{j}\Big\vert\int_{2O_{j}} f\Big\vert\bigg) \\ &\lesssim & \Vert J_{\psi^{-1}}\Vert_{L^{\infty}}\Big(\Vert J_{\psi}\Vert_{L^{\infty}}\Vert f\Vert_{\BMO}+\hbox{I}_1+\hbox{I}_2\Big),
\end{eqnarray*}
where
$$
\hbox{I}_1 \triangleq\frac{1}{\vert B\vert}\sum_{j\textbackslash r_j> \frac{1}{4}}\Big\vert\int_{2O_{j}}f\Big\vert,
$$
and 
$$
\hbox{I}_2 \triangleq\frac{1}{\vert B\vert}\sum_{j\textbackslash r_j\leq \frac{1}{4}}\vert O_{j}\vert\Big\vert\fint_{2O_{j}}f\Big\vert.
$$
H\"{o}lder inequality implies that 
\begin{eqnarray*}
\hbox{I}_1 &\leq &  \frac{1}{\vert B\vert}\sum_{j\textbackslash r_j> \frac{1}{4}}\vert O_{j}\vert^{1- \frac{1}{p}}\Vert f\Vert_{L^p}\\ &\lesssim & \frac{1}{\vert B\vert}\sum_{j}\vert O_{j}\vert \Vert f\Vert_{L^p}\\ & \lesssim &\Vert J_{\psi}\Vert_{L^{\infty}} \Vert f\Vert_{L^p}.
\end{eqnarray*}
In order to estimate the term $\hbox{I}_2$, we consider a collection of open balls $(\tilde{O}_j)_j$ such that, for all $j $ in $\NN$,  $\tilde{O}_j$  is  concentric to $O_j$ and of radius equal to 1. Then, as $r_j\leq\frac{1}{4}$ we have  $4O_j\subset \tilde{O}_j$ and  thus using  the definition of the $BMO_F$-norm gives
\begin{eqnarray*}
\hbox{I}_2 &\lesssim & \frac{1}{\vert B\vert}\sum_{j\textbackslash r_j\leq \frac{1}{4}}\vert O_{j}\vert\Big\vert\fint_{2O_{j}} f-\fint_{\tilde{O_{j}}}f\Big\vert +\frac{1}{\vert B\vert}\sum_{j\textbackslash r_j\leq \frac{1}{4}}\vert O_{j}\vert\Big\vert\fint_{\tilde{O_{j}}}f\Big\vert \\ &\lesssim & \frac{1}{\vert B\vert}\sum_{j\textbackslash r_j\leq \frac{1}{4}}\vert O_{j}\vert F(1-\ln 2r_{j})\Vert f\Vert_{\BMOF}+\Vert J_{\psi}\Vert_{L^{\infty}}\Vert f\Vert_{L^{p}}.
\end{eqnarray*}
We set
 $$
 V_k\triangleq \sum_{j\atop\\ e^{-k-1}\leq 4r_j\leq e^{-k}} \vert O_j\vert .
 $$
 Then,
 $$
 \frac{1}{\vert B\vert}\sum_{r_j\leq \frac{1}{4}}\vert O_{j}\vert F(1-\ln 2r_{j})\leq\frac{1}{\vert B\vert}\sum_{k\geq 0}V_{k} F( k+4).
 $$
 Fix $N\in \NN^\star$ and split  the last sum into two parts according to $k\geq N$ and $k>N$ gives
\begin{eqnarray*}
 \frac{1}{\vert B\vert}\sum_{r_j\leq \frac{1}{4}}\vert O_{j}\vert F(1-\ln 2r_{j})&\leq &\frac{1}{\vert B\vert}\sum_{k\leq N}V_{k} F( N+4)+\frac{1}{\vert B\vert}\sum_{k>N}V_{k} F( k+4).\end{eqnarray*}
For $N\geq \beta(t)$, we may use  Lemma \ref{ap} and inequality  (\ref{eqis}) leading to
\begin{eqnarray*}
\frac{1}{\vert B\vert}\sum_{j}\vert O_{j}\vert F(1-\ln 2r_{j})&\lesssim & \Vert J_{\psi}\Vert_{L^{\infty}} \Big(F(N+4)+ \sum_{k>N}e^{-\frac{k}{\beta(t)}}r^{-1}F(k+4)\Big)\\&\lesssim & \Vert J_{\psi}\Vert_{L^{\infty}} \big(1+ e^{-\frac{N}{\beta(t)}}\beta(t)r^{-1}\big)  F(N+4).
\end{eqnarray*}
We choose $N=\big[\beta(t)(\beta(t)-\ln r)\big]$ and by using \eqref{Eq709} we obtain
$$
\frac{1}{\vert B\vert}\sum_{j}\vert O_{j}\vert F(1-\ln 2r_{j})\lesssim \Vert J_{\psi}\Vert_{L^{\infty}} F\big(1+\beta(t)^2+\beta(t)\big\vert \ln \Vert J_{\psi}\Vert_{L^{\infty}}\big\vert\big).
$$
Putting together the previous estimates gives
\begin{eqnarray*}
 \fint_{B}\Big\vert f\circ\psi -\fint_{B}f\circ \psi\Big\vert &\leq &\|J_{\psi^{-1}}\|_{L^\infty}\| J_\psi\|_{L^\infty}\|f\|_{BMO_F\cap L^p} 
 F\big(\beta(t)^2+\beta(t)\big\vert \ln \Vert J_{\psi}\Vert_{L^{\infty}}\big\vert\big)\\
 &\le& C e^{C\|\textnormal{div }v\|_{L^1_tL^\infty}}\|f\|_{BMO_F\cap L^p} F\Big(e^{C\|v\|_{L^1_t LL}}\|\textnormal{div}\,v\|_{L^1_tL^\infty}\Big) .
 \end{eqnarray*} 
Combining this estimate with \eqref{conv12} we obtain
\begin{eqnarray*}
\Vert f\circ\psi\Vert_{\BMO} &\lesssim & C \|f\|_{BMO_F\cap L^p} e^{C\|\textnormal{div }v\|_{L^1_tL^\infty}} F\Big(e^{C\|v\|_{L^1_t LL}}\|\textnormal{div}\,v\|_{L^1_tL^\infty}\Big) .
\end{eqnarray*}
In view of the polynomial growth condition of $F$ seen in  (\ref{eq}) we get
\begin{eqnarray*}
\Vert f\circ\psi\Vert_{\BMO} &\le &C \Vert f\Vert_{\BMO_F \cap L^p}e^{C\int_0^t \Vert\textnormal{div}v(\tau)\Vert_{L^{\infty}}d\tau}F\big(e^{C\|v\|_{L^1_t LL}}\big).
\end{eqnarray*}

Now we shall move to the treatment of the second part of the $BMO_F$ norm. 
\vspace{0,3cm}

$\bullet$ {\bf Step 2: Estimate of the second part of the $BMO_F$ norm}.  This will be done in a similar way to the first part.
Denote $B_i=B(x_i,r_i)$ and $\tilde{B}_i=B(x_i,g_{\psi}(r_i))$ for $i\in \{1, 2\}$, with $2B_2\subset B_1$ and $r_1<1$. Set
\begin{eqnarray*}
J &\triangleq &\frac{\big\vert\fint_{B_{2}}f\circ\psi-\fint_{B_{1}}f\circ\psi\big\vert}{F\Big(\dfrac{1-\ln r_{2}}{1-\ln r_{1}}\Big)},
\end{eqnarray*}
  We have three cases to discuss:\\
\underline{\textbf{Case 1}} : $r_1\leq(4e)^{-\beta(t)}\min\big\{1,\frac{1}{\Vert J_{\psi}\Vert_{L^{\infty}}}\big\}$.
Since the denominator of the quantity $J$ is larger than one, we may write
\begin{eqnarray*}
J&\leq & J_{1}+J_{2}+J_{3}.
\end{eqnarray*}
Where
\begin{eqnarray*}
J_1 & \triangleq &\Big\vert\fint_{B_{2}}f\circ\psi-\fint_{\tilde{ B_{2}}}f\Big\vert + \Big\vert\fint_{B_{1}}f\circ\psi-\fint_{\tilde{ B_{1}}}f\Big\vert, \\
J_2 & \triangleq &\frac{\big\vert\fint_{\tilde{B}_{2}}f-\fint_{2\tilde{B}_{1}}f\big\vert}{F\Big(\dfrac{1-\ln r_{2}}{1-\ln r_{1}}\Big)}, \\
J_3 & \triangleq &\Big\vert\fint_{\tilde{B}_{1}}f-\fint_{2\tilde{B}_{1}}f\Big\vert. \\
\end{eqnarray*}
The treatment of $J_1$ will be exactly the same as for the case 1 from step 1. For $J_2$, by definition of the second part of the $BMO_F$ norm, we have 
$$
J_2\leq \frac{F\Big(\frac{1-\ln g_{\psi}(r_2)}{1-\ln g_{\psi}(r_1)}\Big)}{F\Big(\dfrac{1-\ln r_{2}}{1-\ln r_{1}}\Big)}\Vert f\Vert_{\BMOF}.
$$
According to the inequality (\ref{eqg}), we get 
\begin{eqnarray*}
F\Big(\frac{1-\ln g_{\psi}(r_2)}{1-\ln g_\psi(r_1)}\Big)&= & F\Big(\frac{1-\ln r_1}{1-\ln g_{\psi}(r_1)}\quad\frac{1-\ln r_2}{1-\ln r_1}\quad\frac{1-\ln g_{\psi}(r_2)}{1-\ln r_2}\Big)\\ &\lesssim & F\Big((1+\beta(t))^{2}\dfrac{1-\ln r_{2}}{1-\ln r_{1}}\Big).
\end{eqnarray*}
 This gives  in view of Definition \ref{def1},
$$
J_2 \lesssim F\big(\beta^2(t)\big)\Vert f\Vert_{\BMOF}.
 $$
Concerning $J_3$ we use the inequality (\ref{eq1}) to get
$$
J_3\lesssim \Vert f\Vert_{BMO}.  
$$ 
\vspace{0,3cm}

\underline{\textbf{Case 2}} : $r_2\geqslant (4e)^{-\beta(t)}\min\big\{1,\frac{1}{\Vert J_{\psi}\Vert_{L^{\infty}}}\big\}$.
As $F$ is larger than 1, we write
 $$
 J\leq\fint_{B_{2}}\big\vert f\circ\psi\big\vert+\fint_{B_{1}}\big\vert f\circ\psi\big\vert.
 $$
 Both terms can be handled as in the analysis of the case 2 from step 1.
 \vspace{0,3cm}
 
\underline{\textbf{Case 3}} : $r_1\geqslant (4e)^{-\beta(t)}\min\big\{1,\frac{1}{\Vert J_{\psi}\Vert_{L^{\infty}}}\big\}$ and $r_2\leq (4e)^{-\beta(t)}\min\big\{1,\frac{1}{\Vert J_{\psi}\Vert_{L^{\infty}}}\big\}$. We decompose $J$ as follows:
 \begin{eqnarray*}
 J &\leq & \fint_{B_{2}}\Big\vert f\circ\psi-\fint_{\tilde{ B_{2}}}f\Big\vert +\frac{\big\vert\fint_{\tilde{B}_{2}}f\big\vert}{F\Big(\dfrac{1-\ln r_{2}}{1-\ln r_{1}}\Big)} +\fint_{B_{1}}\big\vert f\circ\psi\big\vert \\ &\triangleq & J_1+J_2+J_3.
 \end{eqnarray*}
Let $\tilde{B}'_2$ the ball of center $\psi(x_2)$ and radius $1$. Then 
 \begin{eqnarray*}
 J_2 &\leq & \frac{\vert\fint_{\tilde{B}_{2}}f-\fint_{\tilde{B}'_2}f\vert+\vert \fint_{\tilde{B}'_2}f\vert}{F\Big(\dfrac{1-\ln r_{2}}{1-\ln r_{1}}\Big)} \\ & \leq & \frac{F(1-\ln g_\psi(r_2))}{F\Big(\dfrac{1-\ln r_{2}}{1-\ln r_{1}}\Big)}\Vert f\Vert_{\BMOF}+\Vert f\Vert_{L^{p}}.
\end{eqnarray*}
From the Definition \ref{def1} combined  with the  inequality (\ref{eqg}), we obtain
\begin{eqnarray*} 
  \frac{F(1-\ln g_\psi(r_2))}{F\Big(\dfrac{1-\ln r_{2}}{1-\ln r_{1}}\Big)} &= &  \dfrac{F\Big(\dfrac{1-\ln g_\psi(r_{2})}{1-\ln r_{2}}\quad \dfrac{1-\ln r_2}{1-\ln r_1}\quad 1-\ln r_1\Big)}{F\Big(\dfrac{1-\ln r_{2}}{1-\ln r_{1}}\Big)}\\ &\lesssim & \frac{F\Big(\dfrac{1-\ln r_{2}}{1-\ln r_{1}}\quad (1+\beta(t))(1-\ln r_1)\Big)}{F\Big(\dfrac{1-\ln r_{2}}{1-\ln r_{1}}\Big)}\\ &\lesssim & F \big((1+\beta(t)(\beta+|\ln \|J_\psi\|_{L^\infty}|)\\
  &\lesssim & F(e^{C\|v\|_{L^1_t LL}}) F \big(\|\textnormal{div} v\|_{L^1_tL^\infty}).
 \end{eqnarray*}
The terms $J_1$ and $J_3$ can be handled in the same way as the cases  1 and 2 from step 1.\\
The proof is now achieved.\\
\end{proof}
Our next task is to study the composition  in  the space $\LMOF\cap L^\infty$. This will  more easier than the the $BMO_F$ space since we shall  use in a crucial way  the $L^\infty$ norm. 

\begin{proposition}\label{prop2}
There exists a positive constant $C$ such that, for any function $f\in \LMOF\cap L^\infty$ and for $\psi$ a flow associated to a smooth vector field $v$, we have

$$
\Vert( f\circ\psi)(t)\Vert_{\LMOF}\leq C F\big(e^{C\Vert v\Vert_{L^1_tLL}}\big)e^{C\Vert \textnormal{div}\, v\Vert_{L^1_tL^\infty}}\Vert f \Vert_{LMO_F\cap L^{\infty}}.
$$
\end{proposition}
\begin{proof}
Identically to the proof of the Theorem \ref{p2}, we will proceed in two steps; the first one concerns the first term of the norm and the second one is devoted to the other term.\\
\vspace{0,2cm}

$\bullet$ \textbf{Step1: Estimate of the first part of the norm.} One distinguishes two cases:\\

\underline{\textbf{Case 1}} : $r\leq (4e)^{-\beta(t)}.$
We may write
\begin{eqnarray*}
F(1- \ln r)\fint_{B}\Big\vert f\circ\psi-\fint_{B}f\circ \psi\Big\vert \dx  &\leq & 2F(1- \ln r)\fint_{B}\Big\vert f\circ \psi -\fint_{\tilde{B}}f\Big\vert \dx. 
\end{eqnarray*}
Recall that $\tilde{B}$ is the ball of center $\psi(x_0)$ and radius $g_\psi(r).$
A change of variable gives
\begin{eqnarray*} 
F(1- \ln r)\fint_{B}\Big\vert f\circ\psi-\fint_{B}f\circ \psi\Big\vert \dx &\lesssim &\Vert J_{\psi^{-1}}\Vert_{L^{\infty}}\frac{F(1- \ln r)}{\vert B\vert}\int_{\psi(B)} \Big\vert f -\fint_{\tilde{B}}f\Big\vert dx. 
\end{eqnarray*}
Using the Whitney covering lemma used in the proof of  the Theorem \ref{p2} we get
\begin{eqnarray*}
F(1- \ln r)\fint_{B}\Big\vert f\circ\psi-\fint_{B}f\circ \psi\Big\vert \dx &\lesssim & \Vert J_{\psi^{-1}}\Vert_{L^{\infty}}\frac{F(1- \ln r)}{\vert B\vert}\sum_{j}\vert O_{j}\vert \fint_{2O_{j}}\Big\vert f-\fint_{\tilde{B}}f\Big\vert  \\ &\lesssim & \Vert J_{\psi^{-1}}\Vert_{L^{\infty}} \big( \hbox{I}_{1}+\hbox{I}_{2}\big),
\end{eqnarray*}
 where
$$
I_{1}\triangleq\frac{F(1- \ln r)}{\vert B\vert}\sum_{j}\vert O_{j}\vert \fint_{2O_{j}}\Big\vert f-\fint_{2O_{j}}f\Big\vert,
$$
and
$$
\hbox{I}_{2}\triangleq \frac{F(1- \ln r)}{\vert B\vert}\sum_{j}\vert O_{j}\vert\Big\vert\fint_{2O_{j}}f-\fint_{\tilde{B}}f\Big\vert.
$$
 In view of the the polynomial growth property of $F$, we have
\begin{eqnarray*}
\hbox{I}_{1} &=&\frac{1}{\vert B\vert}\sum_{j}\vert O_{j}\vert F\Big(\frac{1- \ln r}{1- \ln 2r_j}\big(1- \ln (2r_j)\big)\Big)\fint_{2O_{j}}\Big\vert f-\fint_{2O_{j}}f\Big\vert \\ &\lesssim & \frac{1}{\vert B\vert}\sum_{j}\vert O_{j}\vert F\Big(\frac{1- \ln r}{1- \ln( 2r_j)}\Big)F(1- \ln(2r_j))\fint_{2O_{j}}\Big\vert f-\fint_{2O_{j}}f\Big\vert \\ &\lesssim & \frac{1}{\vert B\vert}\sum_{j}\vert O_{j}\vert F\Big(\frac{1- \ln r}{1- \ln(2r_j)}\Big)\Vert f\Vert_{\LMOF} 
\end{eqnarray*}
As $r_j\leq g_\psi(r)$ and by \eqref{eqg} one has
\begin{eqnarray*}
\frac{1- \ln r}{1- \ln 2r_j}&\lesssim &\frac{1- \ln r}{1- \ln 2g_\psi(r)}\\&\lesssim & 1+\beta(t).
\end{eqnarray*}
Consequently we get in view of \eqref{eqn},
$$
\hbox{I}_{1}\lesssim\Vert J_{\psi}\Vert_{L^{\infty}}F\big(\beta(t)\big)\Vert f \Vert_{\LMOF}.
$$
Since $4O_{j}\subset \tilde{B}$, $g_\psi(r)\leq 1$ and by the definition of the second part of the $\BMOF$-norm, $\hbox{I}_2$ can be estimated as follows  
\begin{eqnarray*}
\hbox{I}_{2} &\lesssim &\frac{1}{\vert B\vert}\sum_{j}\vert O_{j}\vert F\Big(\frac{1-\ln r}{1-\ln g_{\psi}(r)}\big(1-\ln g_\psi (r)\big)\Big)  \Big\vert\fint_{2O_{j}}f-\fint_{\tilde{B}}f\Big\vert \\ &\lesssim & \frac{1}{\vert B\vert}\sum_{j}\vert O_{j}\vert F\big(1+\beta(t)\big) F(1-\ln g_{\psi}(r))\Big\vert\fint_{2O_{j}}f-\fint_{\tilde{B}}f\Big\vert \\ & \lesssim & \Vert J_{\psi}\Vert_{L^{\infty}}F\big(\beta(t)\big)\Vert f \Vert_{\LMOF}.
\end{eqnarray*}

\underline{\textbf{Case 2}} : $1\geqslant r\geq (4e)^{-\beta(t)}.$ Under this assumption we have  $\vert\ln r\vert\lesssim \beta(t)$ and  then we can immediately  deduce that
\begin{eqnarray*}
F(1- \ln r)\fint_{B}\Big\vert f\circ\psi -\fint_{B}f\circ \psi\Big\vert   &\lesssim & F\big(1+\beta(t)\big)\fint_{B}\Big\vert f\circ\psi\Big\vert  \\ &\lesssim & F\big(\beta(t)\big)\Vert f\Vert_{L^{\infty}}. 
\end{eqnarray*}
\vspace{0,2cm}

\textbf{Step 2: Estimate of the second part of the norm.}
Denote $B_i=B(x_i,r_i)$ and $\tilde{B}_i=(x_i,g_{\psi}(r_i))$ for $i\in \{1, 2\} $ with $2B_2\subset B_1$ and $r_1<1$. We shall estimate $J$ defined by,
\begin{eqnarray*}
J &\triangleq & F(1-\ln r_{1})\Big\vert\fint_{B_{2}}f\circ\psi-\fint_{B_{1}}f\circ\psi\Big\vert .
\end{eqnarray*}
There are two cases to discuss depending on the size of $r_1$. \\
\underline{\textbf{Case 1}} : $r_1\leq (4e)^{-\beta(t)}$.
We write
\begin{eqnarray*}
J  &\leq & J_{1}+J_{2}+J_{3},
\end{eqnarray*}
where
\begin{eqnarray*}
J_1 & \triangleq &F(1-\ln r_{1})\Big(\Big\vert\fint_{B_{2}}f\circ\psi-\fint_{\tilde{ B_{2}}}f\Big\vert + \Big\vert\fint_{B_{1}}f\circ\psi-\fint_{\tilde{ B_{1}}}f\Big\vert\Big) \\ 
J_2 & \triangleq &F(1-\ln r_{1})\Big\vert\fint_{\tilde{B}_{2}}f-\fint_{2\tilde{B}_{1}}f\Big\vert \\
J_3 & \triangleq & F(1-\ln r_{1})\Big\vert\fint_{\tilde{B}_{1}}f-\fint_{2\tilde{B}_{1}}f\Big\vert.
\end{eqnarray*}
Reproducing the same arguments as for  the  case 1 from step 1 we can estimate $J_1$. For $J_2$,  we use  the polynomial growth property  of $F$ with the inequality \eqref{eqg} to get
\begin{eqnarray*}
J_2 &\leq & F\Big(\frac{1-\ln r_1}{1-\ln 2g_{\psi}(r_1)}\big(1-\ln 2g_{\psi}(r_1)\big)\Big)\Big\vert\fint_{\tilde{B}_{2}}f-\fint_{2\tilde{B}_{1}}f\Big\vert \\ &\leq & F\Big(\frac{1-\ln r_1}{1-\ln 2g_{\psi}(r_1)}\Big)F\big(1-\ln 2g_{\psi}(r_1)\big)\Big\vert\fint_{\tilde{B}_{2}}f-\fint_{2\tilde{B}_{1}}f\Big\vert \\ &\leq & F\big(1+\beta(t)\big)\Vert f\Vert_{\LMOF}.
\end{eqnarray*}
Similarly, 
\begin{eqnarray*}
J_3& \leq & F\Big(\frac{1-\ln r_1}{1-\ln 2g_{\psi}(r_1)}\Big)F\big(1-\ln 2g_{\psi}(r_1)\big)\Big\vert\fint_{\tilde{B}_{1}}f-\fint_{2\tilde{B}_{1}}f\Big\vert \\ &\leq & F\big(1+\beta(t)\big)\Vert f\Vert_{\LMOF}. 
\end{eqnarray*}
\underline{\textbf{Case 2}} : $r_1\geqslant (4e)^{-\beta(t)}$. Since $F$ is non-decreasing then,
\begin{eqnarray*}
 J&\leq& F(1-\ln r_{1})\Big(\fint_{B_{2}}\big\vert f\circ\psi\big\vert+\fint_{B_{1}}\big\vert f\circ\psi\big\vert\Big)\\
 &\lesssim& F(\beta(t))\|f\|_{L^\infty}.
\end{eqnarray*}
This completes the proof Proposition \ref{prop2}.
 \end{proof}
 Now we have all the necessary ingredients for the proof of  Theorem \ref{th2}
\subsection{Proof of Theorem 2}
We set $ g(t,x)=f(t,\psi(t,x))$. Then, in view of (\ref{12}), we see that $g$ satisfies the following equation
$$
\partial_t g(t,x)+(\textnormal{div}\, v)(t,\psi(t,x))g(t,x)=0, \quad g(0,x)=f_0(x).
$$
It follows that
\begin{eqnarray*}
g(t,x)&=&f_0(x)e^{-\int_0^t (\textnormal{div}\, v)(\tau,\psi(\tau))d\tau}\\
&=& f_0(x)+f_0(x)\Big( e^{-\int_0^t (\textnormal{div}\, v)(\tau,\psi(\tau))d\tau}-1\Big).
\end{eqnarray*}
According to the law product stated in Proposition \ref{prod}, we have
\begin{equation*}
\Vert g(t)\Vert_{\BMOF\cap L^p}\leq C \Vert f_0\Vert_{\BMOF\cap L^p}\Big(1+\big\Vert e^{-\int_0^t (\textnormal{div}\, v)(\tau,\psi(\tau))d\tau}-1\big\Vert_{\LMOF\cap L^\infty}\Big).
\end{equation*}
Therefore by applying the assertion (ii) of  Proposition \ref{alg} to the function $x\longmapsto e^x-1$ , we get
$$
\Vert g(t)\Vert_{\BMOF\cap L^p}\leq C \Vert f_0\Vert_{\BMOF\cap L^p}\Big(1+e^{C\Vert \textnormal{div}\,  v\Vert_{L^1_tL^{\infty}}} \int_0^t\Vert (\textnormal{div}\, v)(\tau,\psi(\tau))\Vert_{\LMOF\cap L^\infty}d\tau\Big).
$$
Furthermore, according to Proposition \ref{prop2} we infer that
$$
\Vert g(t)\Vert_{\BMOF\cap L^p}\leq C\Vert f_0\Vert_{\BMOF\cap L^p}\Big(1+e^{C\Vert \textnormal{div}\, v\Vert_{L^1_tL^{\infty}}} F\Big(e^{C\Vert v\Vert_{L^1_tLL}}\Big)\Vert\textnormal{div}\, v\Vert_{L^1_t(\LMOF\cap L^\infty)}\Big).
$$ 
Finally,  Theorem \ref{th1} gives
\begin{eqnarray*}
\Vert f(t)\Vert_{\BMOF\cap L^p}\leq C\Vert f_0\Vert_{\BMOF\cap L^p}e^{C\Vert \textnormal{div}\, v\Vert_{L^1_tL^{\infty}}}F\Big(e^{C\Vert v\Vert_{L^1_tLL}}\Big)\Big(1+ F\big(e^{C\Vert v\Vert_{L^1_TLL}}\big)\Vert\textnormal{div}\, v\Vert_{L^1_t(\LMOF\cap L^\infty)}\Big).
\end{eqnarray*}
This completes the proof  of the theorem. 
\section{Some classical estimates}
 The aim of this section is to highlight two useful estimates for the system (\ref{EC}), those will be of great importance for obtaining a lower bound of the lifespan of
the solution $v_\varepsilon$. In the first instance we shall recall  classical energy estimates for the full system, afterwards, we lead a short discussion on  Strichartz estimates for the wave operator.
\subsection{Energy estimates}
The following energy estimates are  classical and for the proof we refer the reader  for instance to  \cite{D-H,H,K-M2}.
\begin{proposition}\label{enr}
Let $(v_\varepsilon , c_\varepsilon)$ be a smooth solution of $\eqref{EC}$. For $s>0$ there exists a constant $C>0$ such that 
$$
\Vert (v_\varepsilon ,c_\varepsilon )(t)\Vert_{H^s}\leq C\Vert(v_{0,\varepsilon},c_{0,\varepsilon})\Vert_{H^s}e^{CV_{\varepsilon}(t)}, \quad \forall t\geq 0
$$
with
$$
V_\varepsilon(t)\triangleq\Vert \nabla v_\varepsilon\Vert_{L_t^1L^\infty}+\Vert\nabla c_\varepsilon\Vert_{L_t^1L^\infty}.
$$
\end{proposition}

\subsection{ Strichartz estimates}
The main interest of using Strichartz estimates is to deal with the ill-prepared data in the presence of the singular terms in $\frac1\varepsilon$. Actually, it has been shown that the average in time of the compressible and the acoustic parts, which are governed by a coupling non linear wave equations, disappear when the Mach number approaches zero. The details of this assumption has been discussed for instance in \cite{H} and \cite{H-S} for initial data in Besov spaces, but for the convenience of the reader we briefly outline the main arguments of the proof. The system (\ref{EC}) can be rewritten under the form
\begin{displaymath}
\left\{ \begin{array}{ll}
\partial_{t}v_\varepsilon+\frac{1}{\varepsilon}\nabla c_\varepsilon= -v_\varepsilon.\nabla v_\varepsilon-\bar{\gamma}c_\varepsilon\nabla c_\varepsilon \triangleq f_\varepsilon &\\
\partial_{t}c_\varepsilon+\frac{1}{\varepsilon}\textnormal{div}\, v_\varepsilon= -v_\varepsilon.\nabla c_\varepsilon-\bar{\gamma}c_\varepsilon\textnormal{div}\, v_\varepsilon \triangleq g_\varepsilon &\\
(v_\varepsilon ,c_\varepsilon)_{| t=0}=(v_{\varepsilon ,0},c_{\varepsilon ,0}).
\end{array} \right. 
\end{displaymath} 
We denote by $\Q v_\varepsilon \triangleq\nabla \Delta^{-1}\textnormal{div}\, v_\varepsilon$ the compressible part of the velocity $v_\varepsilon$. Then the complex-valued functions
$$
\Gamma_\varepsilon\triangleq\Q v_\varepsilon -i\nabla\vert\mathrm{D}\vert^{-1}c_\varepsilon \quad \textnormal{and}\quad \Upsilon_\varepsilon\triangleq \vert \mathrm{D}\vert^{-1}\textnormal{div}\, v_\varepsilon +i c_\varepsilon
$$
satisfy the following wave equations
\begin{equation}\label{st1}
(\partial_t +\frac{i}{\varepsilon}\vert \mathrm{D}\vert)\Gamma_\varepsilon =\Q f_\varepsilon -i\nabla \vert \mathrm{D}\vert^{-1}g_\varepsilon
\end{equation}
and
\begin{equation}\label{st2}
(\partial_t +\frac{i}{\varepsilon}\vert \mathrm{D}\vert)\Upsilon_\varepsilon =\vert \mathrm{D}\vert^{-1}\textnormal{div}\, f_\varepsilon -ig_\varepsilon
\end{equation}
with $\vert \mathrm{D}\vert=(-\Delta)^{\frac{1}{2}}$. \\

Now we can  following Strichartz estimates whose proof can be found for instance in \cite{B-C-D,D-H,G}.
\begin{lemma}\label{lem2}
Let $\varphi$ be a solution of the wave equation 
$$
(\partial_t+\frac{i}{\varepsilon}\vert \mathrm{D}\vert^{-1})\varphi=G, \quad \varphi_{\vert t=0}=\varphi_0.
$$
Then, there exists a constant $C>0$ such that for all $T>0$ and  $2<p\leq+\infty$,
$$
\Vert\varphi\Vert_{L_{T}^{r}L^{p}}\leq C\varepsilon^{\frac{1}{4}-\frac{1}{2p}}\Big(\Vert \varphi_0\Vert_{\dot{B}_{2,1}^{\frac{3}{4}-\frac{3}{2p}}}+\int_{0}^{t}\Vert G(\tau)\Vert_{\dot{B}_{2,1}^{\frac{3}{4}-\frac{3}{2p}}}d\tau\Big),
$$
with $r=4+\frac{8}{p-2}$.
\end{lemma}
As an application we get the following result.
\begin{coro}\label{st}
Let $s>0$  and  $(v_{0,\varepsilon},c_{0,\varepsilon})$ be a family in $H^{2+s}$. Then any solution of (\ref{EC})
defined in the time interval $[0, T]$ satisfies 
\begin{equation}\label{estq}
\Vert (\Q v_\varepsilon,c_\varepsilon)\Vert_{L_T^{4}L^{\infty}}\leq C_0^\varepsilon\varepsilon^{\frac{1}{4}}(1+T)e^{CV_{\varepsilon}(T)}.
\end{equation}
Moreover, there exists a positive real number $\eta$ which depends only on $s$ such that 
\begin{equation}
\Vert(\textnormal{div}\, v_\varepsilon,\nabla c_\varepsilon)\Vert_{L_T^{1}B^{s/3}_{\infty,\infty}}\leq C_0^\varepsilon(1+T^{\frac{7}{4}}) \varepsilon^{\eta}e^{CV_\varepsilon(T)}.
\end{equation}
Where
$$
V_\varepsilon(T)\triangleq\Vert \nabla v_\varepsilon\Vert_{L_T^1L^\infty}+\Vert\nabla c_\varepsilon\Vert_{L_T^1L^\infty},
$$
and $C_0^\varepsilon$ depends only on the quantity $\Vert (v_{0,\varepsilon},c_{0,\varepsilon})\Vert_{H^{2+s}}$ and with polynomial growth.
\end{coro}
\begin{proof}
Applying Lemma \ref{lem2} to the equation (\ref{st1}), we get
\begin{eqnarray*}
\Vert \Gamma_\varepsilon\Vert_{L_T^4L^{\infty}}\lesssim \varepsilon^{\frac{1}{4}}\Big(\Vert \Gamma^0_\varepsilon\Vert_{\dot{B}^{\frac{3}{4}}_{2,1}}+\int_{0}^{T}\Vert (\Q f_\varepsilon -i\nabla \vert \hbox{D}\vert^{-1}g_\varepsilon)(\tau)\Vert_{\dot{B}^{\frac{3}{4}}_{2,1}}d\tau\Big).
\end{eqnarray*}
Since $\Q$ and $\nabla\vert \hbox{D}\vert^{-1}$ act continuously on the homogeneous Besov spaces, we get
\begin{eqnarray}\label{ga}
\Vert \Gamma_\varepsilon\Vert_{L_T^4L^{\infty}} &\lesssim &\varepsilon^{\frac{1}{4}}\Big(\Vert (v_{0,\varepsilon},c_{0,\varepsilon})\Vert_{\dot{B}^{\frac{3}{4}}_{2,1}}+\int_0^T\Vert ( f_\varepsilon , g_\varepsilon)(\tau)\Vert_{ \dot{B}^{\frac{3}{4}}_{2,1}}d\tau\Big)\notag\\ &\lesssim &\varepsilon^{\frac{1}{4}}\Big(\Vert (v_{0,\varepsilon},c_{0,\varepsilon})\Vert_{H^{1}}+T\Vert ( f_\varepsilon , g_\varepsilon)\Vert_{L^{\infty}_T H^{1}}\Big).
\end{eqnarray}
Where we have used in the last inequality the fact that $H^{1}\hookrightarrow B^{\frac{3}{4}}_{2,1}\hookrightarrow \dot{B}^{\frac{3}{4}}_{2,1}$.\\ 
To estimate $\Vert ( f_\varepsilon , g_\varepsilon)\Vert_{H^{1}}$ we use the following law product
$$\Vert u\cdot\nabla w\Vert_{ H^{1}}\lesssim\Vert u\Vert_{L^\infty}\Vert w\Vert_{ H^{2}}+\Vert w\Vert_{L^\infty}\Vert u\Vert_{ H^{2}}.$$
Then, by definition of $(f_\varepsilon,g_\varepsilon)$ we have
$$
\Vert ( f_\varepsilon , g_\varepsilon)\Vert_{ H^{1}}  \lesssim  \Vert ( v_\varepsilon , c_\varepsilon)\Vert _{ L^{\infty}}\Vert ( v_\varepsilon , c_\varepsilon)\Vert _{ H^{2}}.
$$
Using the embedding $ H^{2}\hookrightarrow L^\infty$ combined with the energy estimates, we get
\begin{eqnarray*}
\Vert ( f_\varepsilon , g_\varepsilon)\Vert_{L^{\infty}_T H^{1}}  &\lesssim &\Vert ( v_\varepsilon , c_\varepsilon)\Vert ^2_{L^{\infty}_T H^{2}}\\ &\lesssim & \Vert ( v_{0,\varepsilon} , c_{0,\varepsilon})\Vert^2_{ H^{2}}e^{CV_\varepsilon(T)}.
\end{eqnarray*}
Inserting this into the estimate \eqref{ga} we find
\begin{eqnarray*}
\Vert \Gamma_\varepsilon\Vert_{L_T^4L^{\infty}} &\lesssim & \varepsilon^{\frac{1}{4}}\Big(\Vert (v_{0,\varepsilon},c_{0,\varepsilon})\Vert_{H^{2}}+T\Vert ( v_{0,\varepsilon} , c_{0,\varepsilon})\Vert ^2_{ H^{2}}e^{CV_\varepsilon(T)}\Big)\\ &\lesssim & C_0^\varepsilon\varepsilon^{\frac{1}{4}}(1+T)e^{CV_\varepsilon(T)}.
\end{eqnarray*}
As the compressible part of $v_\varepsilon$ is the imaginary part of $\Gamma_\varepsilon$, then
\begin{eqnarray*}
\Vert \Q v_\varepsilon\Vert_{L_T^4L^{\infty}}  &\lesssim & C_0^\varepsilon\varepsilon^{\frac{1}{4}}(1+T)e^{CV_\varepsilon(T)}.
\end{eqnarray*}
By the same manner, we use (\ref{st2}) to prove
\begin{eqnarray*}
\Vert c_\varepsilon\Vert_{L_T^4L^{\infty}}  &\lesssim & C_0^\varepsilon \varepsilon^{\frac{1}{4}}(1+T)e^{CV_\varepsilon(T)}.
\end{eqnarray*}

To prove  the second estimate we use an interpolation procedure between the Strichartz estimates for lower frequencies and the energy estimates for higher frequencies. More precisely, consider $N$ an integer that will be judiciously fixed later. Then, using the embedding $B^{s/3}_{\infty,1}\hookrightarrow B^{s/3}_{\infty,\infty}$, Bernstein inequality and the continuity of the operator $\mathbb{Q}$ on the Lebesgue space $L^2$, we find
\begin{eqnarray*}
\Vert\textnormal{div}\, v_\varepsilon\Vert_{B^{s/3}_{\infty,\infty}}&=& \Vert\textnormal{div}\,\mathbb{Q} v_\varepsilon\Vert_{B^{s/3}_{\infty,\infty}}\\ &\leq &\sum_{q<N}2^{q\frac{s}{3}} \Vert\Delta_q\textnormal{div}\, \mathbb{Q} v_\varepsilon\Vert_{L^\infty}+\sum_{q\geqslant N}2^{q\frac{s}{3}} \Vert\Delta_q\textnormal{div}\,\mathbb{Q} v_\varepsilon\Vert_{L^\infty}\\ &\lesssim &\sum_{q<N}2^{q(\frac{s}{3}+1)}\Vert\Delta_q\mathbb{Q} v_\varepsilon\Vert_{L^\infty}+2^{- N\frac{s}{3}}\sum_{q\geqslant N}2^{q(\frac{2s}{3}+2)}\Vert\Delta_q \mathbb{Q} v_\varepsilon\Vert_{L^2}\\ &\lesssim & 2^{N(\frac{s}{3}+1)}\Vert \mathbb{Q}v_\varepsilon\Vert_{L^\infty}+2^{- N\frac{s}{3}}\Vert v_\varepsilon\Vert_{B_{2,1}^{\frac{2s}{3}+2}}\\ &\lesssim & 2^{N(\frac{s}{3}+1)}\Vert \mathbb{Q}v_\varepsilon\Vert_{L^\infty}+2^{- N\frac{s}{3}}\Vert v_\varepsilon\Vert_{H^{2+s}}.
\end{eqnarray*}
Where we have used in last inequality the embedding  $H^{2+s}\hookrightarrow B_{2,1}^{\frac{2s}{3}+2}$. Integrating in time and combining  (\ref{estq}) with  the energy estimates  we get
\begin{eqnarray*}
\Vert\textnormal{div}\,v_\varepsilon\Vert_{L^1_T B^{s/3}_{\infty,\infty}}& \lesssim &  C_0^\varepsilon 2^{N(\frac{s}{3}+1)} \varepsilon^{\frac{1}{4}} T^{\frac{3}{4}}(1+T)e^{CV_\varepsilon(t)}+ C_0^\varepsilon 2^{-\frac{s}{3} N}Te^{CV_\varepsilon(T)}\\ & \lesssim &  C_0^\varepsilon (1+T^{\frac{7}{4}})\big( 2^{N(\frac{s}{3}+1)} \varepsilon^{\frac{1}{4}}+ 2^{-\frac{s}{3} N}\big)e^{CV_\varepsilon(T)}.
\end{eqnarray*}
By similar computations, we get
\begin{eqnarray*}
\Vert\nabla c_\varepsilon\Vert_{L^1_T B^{s/3}_{\infty,\infty}} \lesssim C_0^\varepsilon (1+T^{\frac{7}{4}})\big( 2^{N(\frac{s}{3}+1)} \varepsilon^{\frac{1}{4}}+ 2^{-\frac{s}{3} N}\big)e^{CV_\varepsilon(T)}.
\end{eqnarray*}
We choose $N$ such that $2^{N(\frac{2s}{3}+1)}\approx \varepsilon^{\frac{-1}{4}}$. This is equivalent to take
$$
N\approx \frac{1}{4(\frac{2s}{3}+1)}\log_2 \varepsilon^{-1}.
$$
Consequently 
\begin{eqnarray*}
\Vert(\textnormal{div }v_\varepsilon,\nabla c_\varepsilon)\Vert_{L^1_t B^{s/3}_{\infty,\infty}}& \lesssim &  C_0^\varepsilon (1+t^{\frac{7}{4}}) \varepsilon^{\frac{s}{4(2s+3)}}e^{CV_\varepsilon(t)}.
\end{eqnarray*}
This ends the proof of the corollary.\\
\end{proof}

We are now in position to prove our main theorem.
 \section{Main result}
In this section, we shall state more general result than Theorem 1, afterwards, the rest of this section will be devoted
to the discussion of the proof. Our result reads as follows.
 \begin{theorem}\label{th4} 
Let $s,\alpha\in]0,1[$, $ p\in]1,2[$ and $F\in\mathcal{F}$. Consider a family  of initial data $\{(v_{0,\varepsilon},c_{0,\varepsilon})_{0<\varepsilon<1}\}$ such that there exists a   positive constant $C$ which does not depend on  $\varepsilon$ and verifying
$$
  \Vert (v_{0,\varepsilon},c_{0,\varepsilon})\Vert_{H^{2+s}} \leq C (\log\varepsilon^{-1})^{\alpha},
  $$
 $$
 \Vert \omega_{0,\varepsilon}\Vert_{L^p\cap \BMOF}\leq C.
 $$
 Then, the system \eqref{EC} admits a unique solution $(v_\varepsilon,c_\varepsilon)\in C\big([0,T_\varepsilon],H^{2+s}\big)$ with the alternative:
  \begin{enumerate}
\item [(1)]   If $F$ belongs to the class $\mathcal{F}'$ then the lifespan $T_\varepsilon$ of the solution  satisfies the lower bound: 
$$
  T \geqslant \frac{1}{C_0}\M\Big( (1-\alpha)\ln\ln\varepsilon^{-1}\Big) \triangleq \tilde{T}_\varepsilon.
  $$ 
and the vorticity $\omega_\varepsilon$  satisfies 
\begin{eqnarray} \label{glob}
\forall t\in [0, \tilde{T}_\varepsilon],\quad  \Vert \omega_\varepsilon(t)\Vert_{\BMOF\cap L^p} & \leq &C_0(\M^{-1})'\big(C_0(1+t)\big).
\end{eqnarray}
Where $\M: [0,+\infty[\longrightarrow[0,+\infty[$ is defined by
$$
\M(x)\triangleq \int_{0}^{x}\frac{\dy}{F(e^{Cy})}
$$
and $(\M^{-1})'$ denotes the derivative of $\M^{-1}.$
\item [(2)]  If $F$ belongs to the class $\mathcal{F}\backslash\mathcal{F}'$ then there exists $T_0>0 $ independent of $\varepsilon$ such that for all $t\leq T_0$ we have
\begin{eqnarray} \label{loc}
\Vert \omega_\varepsilon(t)\Vert_{\BMOF\cap L^p} & \leq & C_0.
\end{eqnarray} \end{enumerate}
  Moreover, in both cases the compressible and acoustic parts of the solutions tend to zero: there exists $\eta> 0$ such that  for small $\varepsilon$ and for all $0 \leq T \leq \tilde{T}_\varepsilon$ ( respectively $0 \leq T \leq T_0$)
\begin{equation}\label{eqac}
\Vert(\textnormal{div}\, v_\varepsilon, \nabla c_\varepsilon)\Vert_{L^1_TB^{s/3}_{\infty,\infty}}\lesssim \varepsilon^{\eta/2}.
\end{equation}
Assume in addition that $\lim_{\varepsilon\rightarrow 0}\Vert\omega_{0,\varepsilon} -\omega_0\Vert_{L^p}=0$, for some vorticity $\omega_0\in BMO_F\cap L^p$ associated to a divergence-free vector field $v_0$. Then, the vortices  $(\omega_{\varepsilon})_\varepsilon$ converge strongly to the   solution $\omega$ \mbox{of \eqref{vo}} associated to the  initial data $\omega_0$. More precisely, for all  $t\in\RR_+$  $(0 \leq t \leq T_0$ respectively$)$
$$
\lim_{\varepsilon\rightarrow 0}\Vert\big(\omega_{\varepsilon} -\omega\big)(t)\Vert_{L^q}=0\quad \forall q\in[p,+\infty[.
$$
Furthermore, 
\begin{enumerate}
\item if $F\in\mathcal{F}'$  then for all $t\in \RR_+$
 \begin{eqnarray} \label{glob1}
\Vert \omega(t)\Vert_{\BMOF\cap L^p} & \leq &C_0(\M^{-1})'\big(C_0(1+t)\big) .
\end{eqnarray}
\item if $F\in\mathcal{F}\backslash\mathcal{F}'$ then for all $t< T_0$
 \begin{equation}\label{loc1}
 \Vert \omega(t)\Vert_{L^p\cap \BMOF}\leq C_0.
 \end{equation}
\end{enumerate}
 \end{theorem}
 \begin{remark}
 Theorem $\ref{th4}$ recovers the local and global well-posedness theory of the incompressible Euler system  according to the inequalities $(\ref{glob1})$ and $(\ref{loc1}).$
 \end{remark}
The proof of Theorem $\ref{th4}$ will be divided into two parts: in the first one we estimate  the lifespan of the solutions, thereafter, we discuss in the second part the incompressible limit problem.
 \subsection{Lower bound of the lifespan} We will give an a priori bound of $T_\varepsilon$ and show that the acoustic parts vanish when the Mach number goes to zero.
We denote by
$$
W_\varepsilon(t)\triangleq \Vert v_\varepsilon\Vert_{L^1_tLL}\quad\textnormal{and}\quad V_\varepsilon(t)\triangleq \Vert\nabla v_\varepsilon\Vert_{L^1_tL^\infty}+ \Vert\nabla c_\varepsilon\Vert_{L^1_tL^\infty}.
$$
In view of  Theorem \ref{th2} and using the embedding (iii) from Proposition \ref{alg}, we have
\begin{eqnarray*}
\Vert \omega_\varepsilon(t)\Vert_{\BMOF\cap L^p}\leq C_0F\Big(e^{CW_\varepsilon(t)}\Big)\bigg(1+F\Big(e^{CW_\varepsilon(t)}\Big)  \Vert\textnormal{div}\, v_\varepsilon\Vert_{L^1_tB^{\frac{s}{3}}_{\infty,\infty}}\bigg)e^{C\Vert \textnormal{div}\, v_\varepsilon\Vert_{L^1_tL^{\infty}}}.
\end{eqnarray*}
According to Remark \ref{rq1}, the function $F$ has at most a polynomial growth : $F(x)\lesssim 1+x^\beta$. Also, since $\Vert v_\varepsilon\Vert_{LL}\lesssim\Vert\nabla v_\varepsilon\Vert_{L^\infty}$ we may write 
\begin{eqnarray*}
\Vert \omega_\varepsilon(t)\Vert_{\BMOF\cap L^p} &\leq & C_0F\Big(e^{CW_\varepsilon(t)}\Big)\bigg(1+e^{CV_\varepsilon(t)} \Vert\textnormal{div}\, v_\varepsilon\Vert_{L^1_tB^{\frac{s}{3}}_{\infty,\infty}}\bigg)e^{C\Vert \textnormal{div}\, v_\varepsilon\Vert_{L^1_tL^{\infty}}}.
\end{eqnarray*}
 It follows from Corollary \ref{st} that
\begin{eqnarray} \label{omega1}
\Vert \omega_\varepsilon(t)\Vert_{\BMOF\cap L^p} & \leq &C_0F\Big(e^{C\Vert v_\varepsilon\Vert_{L^1_tLL}}\Big)\Big(1+C_0^\varepsilon (1+t^{\frac{7}{4}}) \varepsilon^{\eta}e^{CV_\varepsilon(t)}\Big)e^{C_0^\varepsilon (1+t^{\frac{7}{4}}) \varepsilon^{\eta}e^{CV_\varepsilon(t)}}\notag\\ &\leq&C_0F\Big(e^{CW_\varepsilon(t)}\Big)e^{C_0^\varepsilon (1+t^{\frac{7}{4}}) \varepsilon^{\eta}e^{CV_\varepsilon(t)}},
\end{eqnarray}
we recall that $C_0^\varepsilon$ is a positive constant depending polynomially  on  $\Vert (v_{0,\varepsilon},c_{0,\varepsilon})\Vert_{H^{2+s}}$.\\
Now  Lemma \ref{22} and the embeddings $B^s_{\infty,\infty}\hookrightarrow BMO\hookrightarrow B^0_{\infty,\infty}$ ensure that
\begin{eqnarray}\label{vll}
\Vert v_\varepsilon\Vert_{LL}  &\lesssim & \Vert \omega_\varepsilon\Vert_{BMO_F\cap L^p}+\Vert \mathbb{Q} v_\varepsilon\Vert_{L^{\infty}}+\Vert\textnormal{div}\, v_\varepsilon\Vert_{B^s_{\infty,\infty}}.
\end{eqnarray}
Integrating in time and using Corollary \ref{st} we get
\begin{eqnarray}\label{wep}
W_\varepsilon(t) &\leq& C\int_{0}^{t}\Vert \omega_\varepsilon (\tau)\Vert_{BMO_F\cap L^p}d\tau+C_0^\varepsilon (1+t^{\frac{7}{4}}) \varepsilon^{\eta}e^{CV_\varepsilon(t)}.
\end{eqnarray}
Setting 
$$
\rho_\varepsilon(t)\triangleq C_0^\varepsilon (1+t^{\frac{7}{4}}) \varepsilon^{\eta}e^{CV_\varepsilon(t)}
$$
and inserting the estimate (\ref{omega1}) into \eqref{wep} give
\begin{eqnarray}\label{f}
W_\varepsilon(t) \leq  C_0\int_{0}^{t}F\big(e^{CW_\varepsilon(\tau)}\big)e^{\rho_\varepsilon(\tau)}d\tau+\rho_\varepsilon(t).
\end{eqnarray}
At this stage we  distinguish two cases depending whether $F\in\mathcal{F'}$ or not.\\

\textbf{(1)} If $F\in\mathcal{F'}$, we fix $T > 0$ an arbitrary real number. So the inequality \eqref{f} becomes
\begin{eqnarray*}
\forall t\in[0,T]\quad W_\varepsilon(t) &\leq & C_0\int_{0}^{t}F\big(e^{CW_\varepsilon(\tau)}\big)e^{\rho_\varepsilon(\tau)}d\tau+\rho_\varepsilon(T).
\end{eqnarray*}
We  introduce the function $\M : [0,+\infty[\longrightarrow[0,+\infty[$ defined by
$$
\M(x)\triangleq \int_{0}^{x}\frac{\dy}{F(e^{Cy})}.
$$
Since $\M$ is a nondecreasing function and $\displaystyle{\lim_{x\rightarrow \infty}\M(x)=+\infty}$ then $\M$ is one-to one and  Lemma \ref{osgood} implies that
\begin{eqnarray*}
\forall t\in[0,T]\quad W_\varepsilon(t) &\leq &\M^{-1}\Big(\M\big(\rho_\varepsilon(T)\big)+ C_0 e^{\rho_\varepsilon(t)}t\Big).
\end{eqnarray*}
Then,
\begin{eqnarray*}
W_\varepsilon(T) &\leq &\M^{-1}\Big(\M\big(\rho_\varepsilon(T)\big)+ C_0e^{\rho_\varepsilon(T)}T\Big).
\end{eqnarray*}
Inserting this estimate into (\ref{omega1}) leads to
\begin{eqnarray} \label{ome}
\Vert \omega_\varepsilon(T)\Vert_{\BMOF\cap L^p} & \leq &C_0F\Big(e^{C\M^{-1}\big(\M\big(\rho_\varepsilon(T)\big)+ C_0e^{\rho_\varepsilon(T)}T\big)}\Big)e^{\rho_\varepsilon(T)}.
\end{eqnarray}
 Now we need the following estimate whose proof is given in the Appendix:
$$
 \Vert \nabla v_\varepsilon(T)\Vert_{L^\infty}\lesssim \Vert (v_{0,\varepsilon},c_{0\varepsilon})\Vert_{H^{s+2}}+\Vert \omega_\varepsilon(T)\Vert_{BMO\cap L^p}V_\varepsilon(T).
 $$

This, combined  with Corollary \ref{st} yield
$$ 
V_\varepsilon(T)\le  \rho_\varepsilon(T)+C_0^\varepsilon T+C\int_{0}^{T}\Vert \omega_\varepsilon(t)\Vert_{BMO\cap L^p}V_\varepsilon(t)dt.
$$
Hence Gronwall's inequality implies that
\begin{eqnarray}\label{v}
 V_\varepsilon(T) \le \big( C_0^\varepsilon T+\rho_\varepsilon(T)\big)\exp\Big(C\int_{0}^{T}\Vert \omega_\varepsilon(t)\Vert_{BMO\cap L^p}dt\Big).
 \end{eqnarray}
Putting   (\ref{ome}) in the last estimate we get                                
 \begin{eqnarray*}
  V_\varepsilon(T)  & \le & \big(C_0^\varepsilon T+\rho_\varepsilon(T))\exp\Big(C_0e^{\rho_\varepsilon(T)}\int_0^TF\big(e^{C\M^{-1}(\M(\rho_\varepsilon(t))+ C_0e^{\rho_\varepsilon(t)}t)}\big)dt\Big).
 \end{eqnarray*}
Assuming $\rho_\varepsilon(T)\leq 1$, which is true at least for small $T$, and using the fact that  $\M^{-1}$ is non decreasing we find 
$$
  V_\varepsilon(T) \lesssim C_0^\varepsilon(1+T) \exp\Big(C_0\int_0^TF\big(e^{C\M^{-1}(C_0(1+t))}\big)dt\Big).
$$
Moreover the inequality \eqref{ome} becomes
\begin{eqnarray*} 
\Vert \omega_\varepsilon(T)\Vert_{\BMOF\cap L^p} & \leq & C_0F\Big(e^{C\M^{-1}\big(C_0(1+T)\big)}\Big).
\end{eqnarray*}
A straightforward computation gives
$$
F\big(e^{C\M^{-1}(x)}\big)=(\M^{-1})'(x).
$$
Then we immediately deduce that
\begin{eqnarray*}
\Vert \omega_\varepsilon(T)\Vert_{\BMOF\cap L^p} & \leq &C_0(\M^{-1})'\big(C_0(1+T)\big),
\end{eqnarray*}
and
$$
  V_\varepsilon(T) \lesssim C_0^\varepsilon(1+T) \exp\Big(\M^{-1}\big(C_0(1+T)\big)\Big)
$$
 So from the condition  $\rho_\varepsilon(T)\leq 1$  we can conclude a lower bound of the lifespans of the  solution. Indeed, we have
\begin{eqnarray*}
\rho_\varepsilon(T)&=&C_0^\varepsilon\varepsilon^{\eta} (1+T^\frac{7}{4})e^{CV_\varepsilon(T)} \\ & \leq & \varepsilon^{\eta}\exp\Big(C_0^\varepsilon (1+T) e^{\M^{-1}\big(C_0(1+T)\big)}\Big).  
 \end{eqnarray*}
 Now let $\alpha(\varepsilon)$ be a function going to $\infty$ as $\varepsilon$ approaches zero and choosing $T$ such that
 $$
 C_0(1+T)=\M(\alpha(\varepsilon))
 $$ 
 Then in order to get $\rho_\varepsilon(T)\le \varepsilon^{\frac{\eta}{2}}$ we should impose the constraint
 $$
 \exp\Big(C_0^\varepsilon \M(\alpha(\varepsilon))e^{\alpha(\varepsilon)}\Big)\le \varepsilon^{-\frac{\eta}{2}}.
 $$
 Since $F(x)\ge 1,$ for $x\geq 1$ then from the definition of $\M$ we infer that
 $$
 \M(x)\le x
 $$
 and thus to get the preceding inequality  it suffices to assume
 $$
  \exp\Big(C_0^\varepsilon e^{\alpha(\varepsilon)}\Big)\le \varepsilon^{-\frac{\eta}{2}}.
 $$
 At this stage we see that one can impose  the following conditions,
 $$C_0^\varepsilon\leq C(\ln\varepsilon^{-1})^\alpha \quad\hbox{and}\quad \alpha(\varepsilon)\approx (1-\alpha) \ln\ln\frac1\varepsilon
 $$ 
 for some $\alpha\in(0,1).$
In particular, from Corollary \ref{st}, we conclude that
\begin{eqnarray}\label{eqac1}
\Vert(\textnormal{div}v_\varepsilon, \nabla c_\varepsilon)\Vert_{(L^1_T\cap L^4_T)B^{\frac{s}{3}}_{\infty,\infty}}+\Vert (\Q v_\varepsilon,c_\varepsilon)\Vert_{(L^1_T\cap L^4_T)L^\infty}\lesssim \varepsilon^{\eta/2}
  \end{eqnarray}
\textbf{(2)} Let  $F\in\mathcal{F}\backslash\mathcal{F}'$ then  we return   to the estimate (\ref{f}) and we assume that $\rho_\varepsilon(t)\leq 1$. Using again the fact that $F$ has at most a polynomial growth, we get
\begin{eqnarray*}
W_\varepsilon(t) \leq  C_0e^{CW_\varepsilon(t)}t+1.
\end{eqnarray*}
Consequently we can find $T_0\in (0,1)$ independent of $\varepsilon$  such that
\begin{eqnarray*}
\forall t\in[0,T_0]\quad W_\varepsilon(t) \leq  2.
\end{eqnarray*}
Plugging this estimate into (\ref{omega1}) leads to
\begin{eqnarray*}
\Vert \omega_\varepsilon(t)\Vert_{\BMOF\cap L^p} \leq C_0.
\end{eqnarray*}
This combined with (\ref{v}) and the constraint on $\rho_\varepsilon$ give 
\begin{eqnarray*}
 V_\varepsilon(T) &\leq& C_0^\varepsilon( T+1)e^{C_0T}\\
 &\lesssim& C_0^\varepsilon.
 \end{eqnarray*}
 Hence,
\begin{eqnarray*}
\rho_\varepsilon(T)&=&C_0^\varepsilon\varepsilon^{\eta} (1+T^\frac{7}{4})e^{CV_\varepsilon(T)} \\ & \lesssim & 
\varepsilon^{\eta} e^{2C_0^\varepsilon}. \end{eqnarray*}
 Choosing $C_0^\varepsilon\leq C(\ln\varepsilon^{-1})^\alpha$,   the last expression will be bounded by $\varepsilon^{\frac{\eta}{2}}$ and thus we find $\rho_\varepsilon(t)\leq 1$.

This concludes the  proof of the first part of Theorem \ref{th4}.\\

\subsection{Incompressible limit}
\begin{proof}
As it has already pointed out, the vorticity $\omega_\varepsilon$ has the following structure
\begin{equation}\label{som}
\omega_\varepsilon(t)=\omega_{0,\varepsilon}(\psi^{-1}_\varepsilon(t))e^{-\int_0^t \textnormal{div}v_{\varepsilon}(\tau,\psi(\tau,\psi_{\varepsilon}^{-1}(t)))d\tau}.
\end{equation}
So the question of the convergence of the vortices $\{\omega_\varepsilon\}$ can be examined through the convergence of the flow maps  $\{\psi_\varepsilon\}$.  In other words, we shall establish  the existence of the particle trajectories $\psi$ as a uniform  limit of a subsequence of  $\{\psi_\varepsilon\}$.
Once this flow is constructed, we can propose a candidate for the solution of the incompressible Euler system given by
\begin{equation}\label{Eul4}
\omega(t,x)=\omega_0(\psi^{-1}(t,x))\quad\hbox{and}\quad v(t)=K\star\omega(t), \quad K(x)=\frac{1}{2\pi}\frac{x^\perp}{|x|^2}\cdot
\end{equation}
At this stage and in order to get a solution for \eqref{EI} we need to show that $\psi$ is the flow map associated to the velocity $v.$ For this goal we develop strong convergence results of the \mbox{velocities $\{v_\varepsilon\}$.}

To begin with, let $T$ and $R$ be two  positive real numbers such tha $T\leq \tilde{T}_\varepsilon$. Then,  for all $t\in[0,T]$ and $x\in \bar{B}(0,R)$ we use the integral equation of the flow $\psi_\varepsilon$ to get  \begin{eqnarray}\label{psi11}
\vert \psi^{-1}_{\varepsilon}(t,x)-x\vert&=& \bigg\vert  \int_0^t v_\varepsilon(\tau,\psi_{\varepsilon}(\tau,\psi^{-1}_{\varepsilon}(t,x)))d\tau\bigg\vert\notag\\ &\leq &\int_0^t \Vert v_\varepsilon(\tau)\Vert_{L^\infty} d\tau \notag\\  &\leq & \Vert \mathbb{P} v_{\varepsilon}\Vert_{L_T^1L^{\infty}} +\Vert\mathbb{Q}v_\varepsilon(\tau)\Vert_{L_T^1 L^\infty}.
\end{eqnarray}
Since the incompressible part $\mathbb{P}v_\varepsilon$ has the same vorticity as the total velocity
$$
\textnormal{curl }\mathbb{P}v_\varepsilon =\textnormal{curl }v_\varepsilon,
$$
and $1<p<2$, the Biot-Savart law implies that  
\begin{equation}\label{BS}
\Vert \mathbb{P}v_\varepsilon(\tau)\Vert_{L^\infty} \lesssim \Vert \omega_\varepsilon(\tau)\Vert_{ L^p\cap L^{2p}}.
\end{equation}
Hence, according to the identity (\ref{som}) we find
\begin{eqnarray*}
\Vert \mathbb{P}v_\varepsilon(\tau)\Vert_{L^\infty}& \lesssim  &\Vert \omega_{0,\varepsilon}(\psi_\varepsilon^{-1}(\tau))\Vert_{L^p\cap L^{2p}}e^{ \Vert\textnormal{div}\, v_{\varepsilon}\Vert_{L^1_TL^\infty}}.
\end{eqnarray*}
Using  a change of variable combined with  Lemma \ref{lem5} and the estimate \eqref{eqac1} we get
\begin{eqnarray}\label{p}
\Vert \mathbb{P}v_\varepsilon(\tau)\Vert_{L^\infty} &\lesssim &\Vert \omega_{0,\varepsilon}\Vert_{L^p\cap L^{2p}}e^{C \Vert\textnormal{div}\, v_{\varepsilon}\Vert_{L^1_TL^\infty}}\notag\\  \nonumber&\lesssim &\Vert \omega_{0,\varepsilon}\Vert_{L^p\cap BMO}e^{C\varepsilon^{\eta/2}}\\
&\le& C_0,
\end{eqnarray}
where we have used  the classical interpolation inequality (\ref{inter}).
Plugging the estimates (\ref{eqac1}) and (\ref{p})  into (\ref{psi11}) we find
\begin{eqnarray}\label{psi1}
\vert \psi^{-1}_{\varepsilon}(t,x)\vert   &\le & C_0T+R+1.
\end{eqnarray}
So the family $\{\psi_\varepsilon^{-1}\}$ is uniformly bounded on every compact $[0,T]\times \bar{B}(0,R)$ and it remains to study  its equicontinuity.
According to the   Lemma \ref{p1} we have
$$
\forall (x,y)\in\bar{B}(0,R)^2,\quad  \vert x-y\vert\leq e^{-\exp\big(\Vert v_\varepsilon\Vert_{L^1_TLL}\big)}\Rightarrow\vert \psi^{-1 }_{\epsilon}(t,x)-\psi^{- 1}_{\epsilon}(t,y)\vert\leq e\vert x-y\vert^{\exp(-\Vert v_\varepsilon\Vert_{L^1_TLL})}.
$$
But estimate \eqref{vll} combined with  \eqref{glob} and \eqref{eqac1} ensures the existence of an explicit time continuous function $\alpha(t)>0$ such that 
$$
\Vert v_\varepsilon\Vert_{L^1_TLL}\leq C_0\alpha(T).
$$
Hence, for all $(x,y)\in\bar{B}(0,R)^2$ with $\vert x-y\vert\leq e^{-\exp(C_0\alpha(T))}$ we have
\begin{equation}\label{phi}
\big\vert \psi^{- 1}_{\epsilon}(t,x)-\psi^{- 1}_{\epsilon}(t,y)\big\vert\leq e\vert x-y\vert^{\exp(-C_0\alpha(T))}.
\end{equation}
Consider the backward particle trajectories  that we denote by $(\phi_{\varepsilon}(s,t,x))_\varepsilon$ and which satisfies,
$$
\phi_\varepsilon(s,t,x)=x-\int_s^tv_\varepsilon(\tau,\phi_\varepsilon(\tau,t,x))d\tau.
$$
Then it is a well-known fact that
$$
\phi_\varepsilon(0,t,x)=\psi^{-1}(t,x)\quad\hbox{and}\quad \phi_{\varepsilon}(0,t_2,x)=\phi_{\varepsilon}\big(0,t_1,\phi_{\varepsilon}(t_1,t_2,x)\big)\quad \forall (t_1,t_2)\in[0,T]^2.
$$ 
Consequently we get  in view of \eqref{phi} 
 \begin{eqnarray*}
\big\vert \psi^{-1}_{\varepsilon}(t_1,x)-\psi^{-1}_{\varepsilon}(t_2,x)\big\vert &=& \big\vert \psi^{-1}_{\varepsilon}(t_1,x)-\psi^{-1}_{\varepsilon}(t_1,\phi_{\varepsilon}(t_1,t_2,x))\big\vert\\ & \leq & e \big\vert x-\phi_{\varepsilon}(t_1,t_2,x)\big\vert^{\exp(-C_0\alpha(T))}\\ &=& e\Big\vert \int_{t_1}^{t_2}v_\varepsilon(\tau,\phi_\varepsilon(\tau,t_2,x))d\tau\Big\vert^{\exp(-C_0\alpha(T))}\\&\leq & e\Big\vert \int_{t_1}^{t_2}\big(\Vert \mathbb{P}v_\varepsilon(\tau)\Vert_{L^\infty}+\Vert\mathbb{Q}v_\varepsilon(\tau)\Vert_{ L^\infty}\big)d\tau\Big\vert^{\exp(-C_0\alpha(T))}.
\end{eqnarray*}
despite that
\begin{equation}\label{Eqlog}
\big\vert x-\phi_{\varepsilon}(t_1,t_2,x)\big\vert\le e^{-\exp(C_0\alpha(T))}.
\end{equation}
It follows from (\ref{eqac1}) and (\ref{p}) that
\begin{eqnarray*}
\big\vert \psi^{-1}_{\varepsilon}(t_1,x)-\psi^{-1}_{\varepsilon}(t_2,x)\big\vert \leq  e\Big\vert C_0 \vert t_1-t_2\vert+C\varepsilon^{\eta/2}\vert t_1-t_2\vert^{3/4}\Big\vert^{\exp(-C_0\alpha(T))}  .
\end{eqnarray*}
By taking $|t_2-t_2|<<1$ we see that the condition \eqref{Eqlog} is satisfied and the preceding estimate is justified. Thus  with \eqref{phi} we find that 
 the family $\{\psi^{- 1}_{\epsilon}\}$ is equicontinuous on every compact $[0,T]\times\bar{B}(0,R)$. Consequently,  the Arzela-Ascoli theorem ensures the existence of a subsequence,  denoted also by  $\{\psi^{-1}_{\epsilon}\}$ and a particle trajectories $\psi^{- 1} $, such that $\{\psi^{- 1}_{\varepsilon}\}$ converges uniformly to $\psi^{- 1}$ on every compact $[0,T]\times \bar{B}(0,R)$. Observe that the subsequence may in principle depend on this compact but we can suppress this dependence by using  Cantor's diagonal argument. 
 
Performing the same analysis as previously to the integral equation of the flow $\phi_\varepsilon$ we can readily obtain that up to an extraction $\{\phi_\varepsilon\}$ converges uniformly in any compact to some continuous function $\phi$. Moreover for any $t, s \in [0, T],$  $\phi(t,s)$ is a homeomorphism with
$$
\phi^{-1}(t,s,x)=\phi(s,t,x),\quad \psi^{-1}(t,x)=\phi(0,t,x),\quad  \psi(t,x)=\phi(t,0,x).
$$ 
In addition,  for all $t \in [0, T],$  $\psi(t)$ is a Lebesgue measure preserving map. More precisely, for $q\in [1,\infty[$  and  $f \in L^q(\RR^2)$ we have
\begin{equation}\label{pm}
\Vert f\circ\psi_t^{\pm 1}\Vert_{L^q}=\Vert f(t)\Vert_{L^q}. 
\end{equation}
Indeed, for all continuous, compactly supported function $f$, $\{f\circ\psi_{t,\varepsilon}^{\pm 1}\}$ converges pointwisely to $f\circ\psi_t^{\pm 1}$. By the uniform boundedness of  $\{\psi_{t,\varepsilon}^{\pm 1}\}$ with  respect to $\varepsilon$, we get from the integral equation,
$$
\vert x\vert \leq \vert \psi_\varepsilon^{\pm 1}(t,x)\vert +C_0T\quad \forall t \in [0,T].
$$
Since $f$ is compactly supported, we conclude the existence of  $M>0$ such that
$$
\textnormal{supp} (f\circ\psi_{t,\varepsilon})\subset B(0,M+C_0 T).
$$
Therefore, by Lebesgue dominated convergence theorem, we get
\begin{equation}\label{lq}
\lim_{\varepsilon\rightarrow 0} \Vert f\circ \psi_{t,\varepsilon}^{\pm 1}-f\circ\psi_t^{\pm 1}\Vert_{L^q}=0.
\end{equation}
In particular,
$$
\lim_{\varepsilon\rightarrow 0} \Vert f\circ \psi_\varepsilon^{\pm 1}\Vert_{L^q}=\Vert f\circ\psi^{\pm 1}\Vert_{L^q}.
$$
On other hand,  a change of variable combined with Lemma \ref{lem5} lead to
\begin{eqnarray*}
\Vert f\Vert_{L^q}  e^{-C\int_0^t \Vert\textnormal{div}v_{\varepsilon}(\tau)\Vert_{L^\infty}d\tau}\lesssim\Vert f\circ \psi_\varepsilon^{\pm 1}\Vert_{L^q} \lesssim \Vert f\Vert_{L^q} e^{C\int_0^t \Vert\textnormal{div}v_{\varepsilon}(\tau)\Vert_{L^\infty}d\tau}.
\end{eqnarray*}
Taking into consideration the estimate (\ref{eqac1}), the passage to the limit in the last inequality gives the identity (\ref{pm}). To finish the proof we use a density argument.

 With this flow $\psi$ in hand we construct $(v,\omega)$ via \eqref{Eul4} and we shall prove some strong convergence results which give in turn that $(v,\omega)$ is  a solution of the incompressible Euler equations. Recall that $\omega_0$ and $(\omega_\varepsilon)_\varepsilon$ belong to $L^q$ for all $q\in[p,+\infty[$ according to the classical interpolation result between Lebesgue and $BMO$ spaces, see  \eqref{inter}.
Then, we shall prove following   convergence result, 
$$
\lim_{\varepsilon\rightarrow 0}\Vert(\omega_{\varepsilon}-\omega)(t)\Vert_{L^q}\quad \forall q\in[p,+\infty[,\quad  \forall t\in [0,T].
$$ 
For this aim we write
\begin{eqnarray*}
\Vert(\omega-\omega_\varepsilon)(t)\Vert_{L^{q}}&=&\Vert\omega_0\circ\psi^{-1}(t)-\omega_{0,\varepsilon}\circ\psi^{-1}_\varepsilon(t)e^{-\int_0^t \textnormal{div}v_{\varepsilon}(\tau,\psi(\tau,\psi_{\varepsilon}^{-1}(t)))d\tau}\Vert_{L^q}\\ &\leq & \hbox{I}_\varepsilon+\hbox{II}_\varepsilon. 
\end{eqnarray*}
Where
$$
\hbox{I}_\varepsilon\triangleq\Vert\omega_0\circ\psi^{-1}(t)-\omega_{0}\circ\psi^{-1}_\varepsilon(t)\Vert_{L^q},
$$
and
$$
\hbox{II}_\varepsilon\triangleq\Vert\omega_0\circ\psi_{\epsilon}^{-1}(t)-\omega^{\epsilon}_{0}\circ\psi^{-1}_\varepsilon(t)e^{-\int_0^t \textnormal{div}v_{\varepsilon}(\tau,\psi(\tau,\psi_{\varepsilon}^{-1}(t)))d\tau}\Vert_{L^q}.
$$

In view of  the equality (\ref{lq}), we can confirm that
  $$
  \lim_{\varepsilon\rightarrow 0}\hbox{I}_\varepsilon=0.
  $$

To estimate $\hbox{II}_\varepsilon$ we make a change of variable and we use Lemma \ref{lem5} to get
\begin{eqnarray*}
\hbox{II}_\varepsilon &\leq &e^{C \Vert\textnormal{div}v_{\varepsilon}\Vert_{L^1_tL^\infty}}\Vert \omega_0-\omega_{0,\varepsilon}e^{-\int_0^t \textnormal{div}v_{\varepsilon}(\tau,\psi(\tau))d\tau}\Vert_{L^{q}}\\ &\leq &e^{C \Vert\textnormal{div}v_{\varepsilon}\Vert_{L^1_tL^\infty}}\Big(\Vert \omega_0\big(1-e^{-\int_0^t \textnormal{div}v_{\varepsilon}(\tau,\psi(\tau))d\tau}\big)\Vert_{L^{q}}+ \Vert \big(\omega_0-\omega_{0,\varepsilon}\big)e^{-\int_0^t \textnormal{div}v_{\varepsilon}(\tau,\psi(\tau))d\tau}\Vert_{L^{q}}\Big)\\ &\leq &e^{C \Vert\textnormal{div}v_{\varepsilon}\Vert_{L^1_tL^\infty}}\Big(\Vert \omega_0\Vert_{L^q} \Vert1-e^{-\int_0^t \textnormal{div}v_{\varepsilon}(\tau,\psi(\tau))d\tau}\Vert_{L^\infty} + \Vert \omega_0-\omega_{0,\varepsilon}\Vert_{L^q}\Big)\\&\leq &e^{C \Vert\textnormal{div}v_{\varepsilon}\Vert_{L^1_tL^\infty}}\Big(\Vert \omega_0\Vert_{L^q}\int_0^t \Vert\textnormal{div}v_\varepsilon(\tau)\Vert_{L^\infty}d\tau + \Vert \omega_0-\omega_{0,\varepsilon}\Vert_{L^q}\Big).\end{eqnarray*}
Where we have used in the last inequality the estimate
$$
\Vert e^{u}-1\Vert_{L^\infty}\leq \Vert u\Vert_{L^\infty}e^{\Vert u\Vert_{L^\infty}}.
$$
Then, from (\ref{eqac}) and   \eqref{inter} we find 
\begin{eqnarray*}
\hbox{II}_\varepsilon &\lesssim &C_0\varepsilon^{\eta}+\Vert \omega_0-\omega_{0,\varepsilon}\Vert_{L^q}\\ &\lesssim &C_0\varepsilon^{\eta}+\Vert \omega_0-\omega_{0,\varepsilon}\Vert_{L^p}^\frac{p}{q}\Vert \omega_0-\omega_{0,\varepsilon}\Vert_{BMO}^{1-\frac{p}{q}}\\ &\lesssim & C_0\varepsilon^{\eta}+C_0\Vert \omega_0-\omega_{0,\varepsilon}\Vert_{L^p}^\frac{p}{q}.
\end{eqnarray*}
Passing to the limit in the last estimate gives the desired result. Now we shall translate these results to the velocities via Biot-Savart law: we get since  
  $1<p<2$ ,
\begin{equation}\label{pv1}
\Vert \big( \mathbb{P} v_\varepsilon-v\big)(t)\Vert_{L^{\infty}}\lesssim \Vert\big(\omega_{\varepsilon}-\omega\big)(t)\Vert_{L^{p}\cap L^{2p}}.
\end{equation} 
Furthermore, by the  classical Hardy-Littlewood-Sobolev inequality,  one has
\begin{equation}\label{pv2}
\Vert \big( \mathbb{P} v_\varepsilon-v\big)(t)\Vert_{L^{r}}\lesssim \Vert\big(\omega_{\varepsilon}-\omega\big)(t)\Vert_{L^{q}}.
\end{equation}
where 
$r\in [\frac{2p}{2-p},+\infty[$ and $q\in[p,+\infty[$. 
Moreover, in view of the  Calder\'on-Zygmund inequality (\ref{c-z})  we have
$$
\Vert\nabla \big( \mathbb{P} v_\varepsilon-v\big)(t)\Vert_{L^{q}}\lesssim \Vert\big(\omega_{\varepsilon}-\omega\big)(t)\Vert_{L^{q}}\quad \forall q\in [p,+\infty[.
$$
Then, the convergence of $(\mathbb{P} v_\varepsilon)$ to $v$ holds true in $W^{1,r}$ for all $r\in[\frac{2p}{2-p},+\infty[$.\\

\quad It remains to show  that $\omega$ is  a solution of \eqref{vo} associated to the initial vorticity $\omega_0.$  But before doing it, we have to verify that $\psi$ is the flow  associated to $v$. Using the preceding convergence and  the uniform convergence of $\{\psi_\varepsilon\}$ and according to the  estimate (\ref{eqac1}), the passage to the limit in the integral equation of the flow, $$
\psi_\varepsilon(t,x)=x+\int_0^t\mathbb{P}v_\varepsilon(\tau,\psi_\varepsilon(\tau,x))d\tau+\int_0^t\mathbb{Q}v_\varepsilon(\tau,\psi_\varepsilon(\tau,x))d\tau,
$$
yields
$$
\psi(t,x)=x+\int_0^t v(\tau,\psi(\tau,x))d\tau.
$$ 
As $v\in LL$ then by the uniqueness  of the flow  associated to $v$ we can confirm the assumption.\\
Next, let $\phi$ be an element of $\mathcal{D}(\RR_+\times\RR^2)$. By definition of $\omega$ and using a change of variable we have
\begin{eqnarray*}
\int_0^\infty \int_{\RR^2}\omega(t,x)\partial_t\phi(t,x)dx dt &=& \int_0^\infty \int_{\RR^2}\omega_0(x)(\partial_t\phi)(t,\psi(t,x))dx dt.
\end{eqnarray*}
But
\begin{eqnarray*}
(\partial_t\phi)(t,\psi(t,x))&=&\partial_t\big(\phi(t,\psi(t,x))\big)-\partial_t\psi(t,x)\cdot\nabla\phi(t,\psi(t,x))\\ &=& \partial_t\big(\phi(t,\psi(t,x))\big)-(v\cdot\nabla\phi)(t,\psi(t,x)).
\end{eqnarray*}
Thus,
\begin{eqnarray*}
\int_0^\infty \int_{\RR^2}\omega(t,x)\partial_t\phi(t,x)dx dt &=&  -\int_{\RR^2}\omega_0(x)\phi(0,x)dx- \int_0^\infty \int_{\RR^2}\omega_0(x)(v\cdot\nabla\phi)(t,\psi(t,x))dx dt\\&=&  -\int_{\RR^2}\omega_0(x)\phi(0,x)dx- \int_0^\infty \int_{\RR^2}\omega(t,x)(v\cdot\nabla\phi)(t,x)dx dt.
\end{eqnarray*}
Hence, $\omega$ verifies the velocity-vorticity weak formulation :
$$
\int_0^\infty \int_{\RR^2}\omega(t,x)(\partial_t\phi+v\cdot\nabla\phi)(t,x)dx dt+\int_{\RR^2}\omega_0(x)\phi(0,x)dx=0.
$$
Moreover, from ($iv$) of Proposition \ref{proptop} and the estimates (\ref{glob}), (\ref{loc}) we immediately deduce (\ref{glob1}) and  (\ref{loc1}). Finally, the uniqueness of the limit can be concluded by uniqueness of the solution of the incompressible Euler system since the velocity $v$ belongs to $L^1_T LL$.
\end{proof}
\section{Appendix}
\begin{lemma}\label{ap}
Let $B$ be a ball of center $0$ and radius $r>0$ and $\psi$ be the flow associated to a given  smooth vector field $v$. Consider  a Whitney covering of the open connected set $\psi(t,B)$ that is a collection of countable open balls $(O_k)_k$ introduced in the proof of Theorem $\ref{th1}$.\\
For all $k\in\NN$ we set
 $$
 U_k\triangleq \sum_{ e^{-k-1}h(r)<r_j\leq e^{-k}h(r)}\vert O_j\vert,
 $$
and
 $$
 V_k\triangleq \sum_{ e^{-k-1}<4r_j\leq e^{-k}}\vert O_j\vert,
 $$
with $h(r)\triangleq r\max\{1,\Vert J_{\psi}\Vert_{L^{\infty}}\}$ and $ J_{\psi}$ is the Jacobian of $\psi$.\\
Then, there exists an absolute constants $C$ such that for all $k\geq \beta(t)$, we have\\
If $r\lesssim e^{-\beta(t)}\min\big\{1,\frac{1}{\Vert J_{\psi}\Vert_{L^{\infty}}}\big\} $, then
 \begin{equation}\label{eq3}
 U_k\leq C\big(1+\Vert J_{\psi}\Vert_{L^{\infty}})^2e^{-\frac{k}{\beta(t)}}r^{1+\frac{1}{\beta(t)}},
 \end{equation}
 For all $r\in \RR_+^*$ ,
 \begin{equation}\label{eq4}
 V_k\leq C\Vert J_{\psi}\Vert_{L^{\infty}}e^{-\frac{k}{\beta(t)}}r.
 \end{equation}
 Where $\beta(t)=\exp\Big(\int_0^t\Vert v(\tau)\Vert_{LL}d\tau$\Big). 
 \end{lemma}
 \textbf{Proof }
By the definition of $U_k$, we have
$$
U_k\leq\Big\vert \Big\{y\in\psi(B) : d(y,\psi(B)^c)\leq Ce^{-k}r\max\{1,\Vert J_{\psi}\Vert_{L^{\infty}}\}\Big\}\Big\vert 
$$
 Since $\vert \psi (A)\vert \leq \vert A\vert \Vert J_{\psi}\Vert_{L^{\infty}}$ for any measurable set $A\subset\RR^2$, we can deduce that
\begin{equation}\label{uk1}
U_k\leq\Big\vert\Big\{ x\in B : d(\psi(x),\psi(B^c))\leq Ce^{-k}r\max\{1,\Vert J_{\psi}\Vert_{L^{\infty}}\}\Big\}\Big\vert \Vert J_{\psi}\Vert_{L^{\infty}}.
\end{equation}
We set
$$
D_k=\Big\{ x\in B : d(\psi(x),\psi(B^c))\leq Ce^{-k}r\max\{1,\Vert J_{\psi}\Vert_{L^{\infty}}\}\Big\}.
$$
According to the fact $d(\psi(x),\psi(B^c))=d(\psi(x),\partial\psi(B))$ and $\partial\psi(B)=\psi(\partial B)$, we can write 
$$
D_k\subset \Big\{x\in B : \exists y\in \partial B : \vert \psi(x)-\psi(y)\vert \leq  Ce^{-k}r\max\{1,\Vert J_{\psi}\Vert_{L^{\infty}}\}\Big\}.
$$
As $r\lesssim e^{-\beta(t)}\min\big\{1,\frac{1}{\Vert J_{\psi}\Vert_{L^{\infty}}}\big\}$, then
$$
e^{-k}r\max\{1,\Vert J_{\psi}\Vert_{L^{\infty}}\}\lesssim e^{-\beta(t)}.
$$
Then, Lemma \ref{p1} applied with $\psi^{-1}$ gives
\begin{eqnarray*}
 D_k &\subset & \Big\{x\in B :\exists y\in\partial B :\vert x-y\vert \leq C e^{1-\frac{k}{\beta(t)}}r^{\frac{1}{\beta(t)}}(1+\Vert J_{\psi}\Vert_{L^{\infty}})\Big\}.
\end{eqnarray*}   
Therefore,
$$
D_k\subset A=\Big\{x\in B : d(x,\partial B) :\vert x-y\vert \leq C e^{1-\frac{k}{\beta(t)}}r^{\frac{1}{\beta(t)}}(1+\Vert J_{\psi}\Vert_{L^{\infty}})\Big\}.
$$
Inserting this into (\ref{uk1}) gives
 $$
 U_k \leq \Vert J_{\psi}\Vert_{L^{\infty}} \vert D_k\vert \lesssim \big( 1+\Vert J_{\psi}\Vert_{L^{\infty}}\big)^2e^{\frac{-k}{\beta(t)}}r^{1+\frac{1}{\beta(t)}}\, 
 $$
 as claimed. 
Reproducing the same procedure as previously with replacing $Ce^{-k}r\max\{1,\Vert J_{\psi}\Vert_{L^{\infty}}\}$ by $Ce^{-k}$ in \eqref{uk1} and considering the fact that $c_0e^{-k}\lesssim e^{-\beta(t)}$, we get the estimation of $V_k$.
\begin{lemma}
Let $(v_\varepsilon, c_\varepsilon)$ be a smooth solution of the compressible Euler system \eqref{EC} and $\omega_\varepsilon$ be  the  vorticity of $v_\varepsilon$. Then there exists a positive constant $C$ such that

 $$
 \Vert \nabla v_\varepsilon(t)\Vert_{L^\infty}\leq C \Big(\Vert \omega_\varepsilon(t)\Vert_{BMO\cap L^p}V_\varepsilon(t)+\Vert (v_{0,\varepsilon},c_{0,\varepsilon})\Vert_{H^{s+2}}\Big),
 $$
with
$$
V_\varepsilon (t)=\int_0^t\big(\|\nabla v_\varepsilon(\tau)\|_{L^\infty}+\|\nabla c_\varepsilon(\tau)\|_{L^\infty}\big) d\tau
$$
 \end{lemma}
  \begin{proof}
According to Bernstein ineguality and the fact that  $\Vert\dot{\Delta v_\varepsilon}_q\Vert_{L^\infty}\sim 2^{-q}\Vert \dot{\Delta}_q \omega\Vert_{L^\infty}$, we have
 \begin{eqnarray*}
 \Vert \nabla v_\varepsilon\Vert_{L^\infty}&\leq & \Vert \Delta_{-1}\nabla v_\varepsilon\Vert_{L^\infty}+\sum_{0\leq q\leq N} \Vert \Delta_q\nabla v_\varepsilon\Vert_{L^\infty}+\sum_{q\geq N} \Vert \Delta_q\nabla v_\varepsilon\Vert_{L^\infty}\\ & \lesssim & \Vert \Delta_{-1}\nabla v_\varepsilon\Vert_{L^p}+\sum_{0\leq q\leq N} \Vert \Delta_q\omega_\varepsilon\Vert_{L^\infty}+\sum_{q\geq N} 2^q\Vert \Delta_q v_\varepsilon\Vert_{L^\infty}\\ & \lesssim & \Vert \omega_\varepsilon\Vert_{L^p}+N\Vert \omega_\varepsilon\Vert_{B_{\infty,\infty}^{0}}+\Vert v_\varepsilon\Vert_{B_{\infty,\infty}^{s+1}}\sum_{q\geq N} 2^{-qs}\\ & \lesssim & N\Vert \omega_\varepsilon\Vert_{\BMOF\cap L^p}+2^{-Ns}\Vert v_\varepsilon\Vert_{H^{s+2}}.
 \end{eqnarray*}
 where we have used in the last inequality the fact that $H^{s+2}\hookrightarrow B_{\infty,\infty}^{s+1}$. 
Then from the energy estimates we deduce that
 $$
 \Vert \nabla v_\varepsilon(t)\Vert_{L^\infty}\lesssim N\Vert \omega_\varepsilon(t)\Vert_{\BMOF\cap L^p}+2^{-Ns}\Vert (v_{0,\varepsilon},c_{0,\varepsilon})\Vert_{H^{s+2}} e^{CV_\varepsilon(t)}.
 $$
Choosing $N$ such that $2^{Ns}\simeq e^{CV_\varepsilon(t)}$, gives the desired result.\\
\end{proof}
 
\end{document}